\documentclass{amsart}
\usepackage{amssymb,amsmath,amsthm,epsf,epsfig,dsfont,bbm}
\usepackage{latexsym}
\usepackage{upref, eucal}
\usepackage[all]{xy}

\newcommand {\nc} {\newcommand}
\newcommand {\enm} {\ensuremath}

\def \d{\delta}

\nc {\bdm} {\begin{displaymath}}
\nc {\edm} {\end{displaymath}}

\newtheorem {theorem} {\bf{Theorem}}[section]
\newtheorem {lemma}[theorem] {\bf Lemma}
\newtheorem {proposition}[theorem] {\bf Proposition}

\newtheorem {corollary}[theorem] {\bf Corollary}
\numberwithin {equation}{section}

\newcommand{\Ou}{\enm{\mathcal{O}}}

\nc{\J}{\enm{\mathcal{J} }}
\nc {\Z} {\enm{\mathbb{Z}}}
\nc {\form}[1] {\enm{\mbox{\underline{for}}}_{#1}}
\nc {\prol}[1] {\enm{\mbox{\underline{prol}}_{{#1}^*}}}

\nc {\stk} {\stackrel}

\newcommand{\map}{\rightarrow}

\newcommand{\inj}{\hookrightarrow}
\newcommand{\dualav}[1]{{#1}^{\vee}}
\newcommand{\dualmod}[1]{{#1}^{\vee}}

\newcommand{\Pn}[2] {\ensuremath{ {\mathbb{P}}^{#1}_{#2}}}
\nc{\Quot}[3]{\enm{ {\mathfrak{Quot}_{ {#1}/{#2}/{#3}}}}}
\nc{\Hilb}[2]{\enm{ {\mathfrak{Hilb}_{ {#1}/{#2}}}}}
\newcommand{\mfrak}[1]{\mathfrak{#1}}

\newcommand{\bb}[1]{\mathbb{#1}}
\newcommand{\mcal}[1]{\mathcal{#1}}

\nc {\Coh}[4] {\ensuremath{H^{#1}(\Pn{#2}{},{#3}({#4}))}}
\nc {\Ch}[3] {\enm{H^{#1}(X_t,{#2}_t({#3}))}}
\nc {\Qphi}[4]{\enm{ {\mathfrak{Quot}^{~#4}_{ {#1}/{#2}/{#3}}}}}
\nc {\Gra}[4]{\enm{ {\mathfrak{Grass}_{#2}({#3},{#4})}}}
\nc {\HomA}[2]{\enm{\mathrm{Hom}_A{#1}{#2}}}
\nc {\tr}{\mathrm{tr}}

\nc {\C}[2]{\enm{\left(\begin{array}{l} {#1} \\ {#2} \end{array} \right)}}
\nc {\mat}[4]{\enm{\left(\begin{array}{ll}{#1} & {#2} \\ {#3} & {#4}
\end{array}\right)}}

\def \mb{\mbox}

 \def \Z{{\mathbb Z}}

   \def \h{\hat{\ }}

\def \d{\delta} \def \bZ{{\mathbb Z}}

  \def \bX{{\bf X}} \def \bH{{\bf H}}

\def \hG{\hat{\mathbb{G}}_{\mathrm{a}}}

\def \oPsi{\overline{\Psi}}

\def \R1{R((q))[q']\h}

\usepackage{xcolor}
\newcommand{\Ppp}[1]{\tilde{\Pi}_{\phi}^{#1}}
\newcommand{\up}{\tilde{u}}

\newcommand{\lam}{\lambda}
\DeclareMathOperator{\Spec}{\mathrm{Spec}}
\DeclareMathOperator{\Spf}{\mathrm{Spf}}
\DeclareMathOperator{\Lie}{\mathrm{Lie}}
\DeclareMathOperator{\rk}{\mathrm{rk}}
\newcommand{\Hom}{\mathrm{Hom}}

\newcommand{\Ext}{\mathrm{Ext}}

\newcommand{\bI}{{\bf I}}
\newcommand{\switt}{s_{\mathrm{Witt}}}

\newcommand{\longlabelmap}[1]{{\,\buildrel #1\over\longrightarrow\,}}
\newcommand{\longmap}{{\,\longrightarrow\,}}
\def\longisomap{{\,\buildrel \sim\over\longrightarrow\,}} 
\def\isomap{{\,\buildrel \sim\over\rightarrow\,}} 

\newcommand{\mff}{\mfrak{f}}
\newcommand{\Wp}{\tilde{W}}
\newcommand{\Pp}{\tilde{\Pi}}
\nc{\bx}{\mathbf{x}}
\nc{\by}{\mathbf{y}}
\nc{\bz}{\mathbf{z}}
\nc{\ba}{\mathbf{a}}
\nc{\Fp}{\tilde{F}}
\nc{\Rp}{\tilde{R}}
\nc{\mlow}{m_{\mathrm{l}}}
\nc{\mup}{m_{\mathrm{u}}}
\nc{\ord}{\mb{ord }}
\nc{\bXp}{\bX_{\mathrm{prim}}}
\nc{\bPsi}{\mathbf{\Psi}}
\nc{\mult}{\mathrm{mult}}
\nc{\mbB}{\mathbbm{B}}
\nc{\mfor}[1]{{#1}^{\mathrm{for}}}
\nc{\Hdr}{\bH^1_{\mathrm{dR}}(A)}

\title{ Isocrystals associated to arithmetic jet spaces of abelian schemes}
\author{James Borger and Arnab Saha}
\date{}
\email{james.borger@anu.edu.au, arnabsaha0930@gmail.com}
\address{Australian National University, Max Planck Institute for Mathematics}

\begin{document}
\maketitle

\begin{abstract}
Using Buium's theory of arithmetic differential characters, 
we construct a filtered $F$-isocrystal 
$\bH(A)_K$ associated to an abelian scheme $A$ over a $p$-adically complete discrete 
valuation ring with perfect residue field. 
As a filtered vector space, $\bH(A)_K$ admits a natural map to the usual de Rham cohomology of $A$, but
the Frobenius operator comes from arithmetic differential theory and is not the same as
the usual crystalline one.
When $A$ is an elliptic curve, we show that $\bH(A)_K$ has a natural
integral model $\bH(A)$, which implies an integral refinement of a result of
Buium's on arithmetic differential characters. 
The weak admissibility of $\bH(A)_K$ depends on the invertibility of an arithmetic-differential 
modular parameter. Thus the Fontaine 
functor associates to suitably generic $A$ a local Galois representation of an
apparently new kind.
\end{abstract}

\section{Introduction}
The theory of arithmetic jet spaces developed by Buium draws 
inspiration from the theory of differential algebra over a function field.
In differential algebra, given a scheme
$E$ defined over a function field $K$ with a derivation $\partial$ on it, 
one can define the jet spaces $J^nE$ for all $n \in
\bb{N}$. They form an inverse system of 
schemes satisfying a universal property with respect to derivations lifting 
$\partial$. The ring of global functions
$\Ou(J^nE)$ can be thought of as the ring of $n$-th order differential functions
on $E$. In the case when $E$ is an elliptic curve, there exists a 
differential function $\Theta \in \Ou(J^2E)$ which is a homomorphism of group 
schemes from $J^2E$ to
the additive group $\bb{G}_a$.  Such a $\Theta$ is an example of a differential
character of order $2$ for $E$ and is known as a Manin character.
Explicitly, if $E$ is given by the 
Legendre equation $y^2=x(x-1)(x-t)$ over $K=\bb{C}(t)$ with derivation 
$\partial = \frac{d}{dt}$,  then
	$$ 
	\Theta(x,y,x',y',x'',y'')= \frac{y}{2(x-t)^2} - \frac{d}{dt}
	\left[2t(t-1)\frac{x'}{y} \right] + 2t(t-1)x'\frac{y'}{y^2}. 
	$$
where $x,y, x',y', x'',y''$ are the induced coordinates of the jet space $J^2E$.
The existence of such a $\Theta$ can be viewed as a consequence of the Picard--Fuchs equation.
Using the derivation $\partial$ on $K$, we can lift any $K$-rational point $P \in E(K)$ 
canonically to $J^2E(K)$,
and this defines a homomorphism $\nabla:E(K) \map J^2E(K)$. We emphasize that
$\nabla$ is merely a map on $K$-rational points and does not come from a map of schemes. The 
composition $\Theta \circ \nabla: E(K) \map \bb{G}_a(K)$ is then a group homomorphism,
and the torsion points of $E(K)$ are contained in its kernel
simply because $\bb{G}_a(K)$ is torsion free. Such a character $\Theta$ was used by Manin
to give a proof of the Lang--Mordell conjecture for abelian varieties over
function fields \cite{M}. Later Buium gave a different proof, using other
methods, but still using the Manin map \cite{bui92}.

The theory of arithmetic jet spaces, as developed by Buium, proceeds similarly. Derivations $\partial$ are
replaced by what are known as $\pi$-derivations $\d$. They naturally arise from the theory of $\pi$-typical Witt
vectors. For instance, when our base ring $R$ is an unramified extension of the ring of $p$-adic integers $\Z_p$
and when $\pi= p$, the Fermat quotient operator $\d x = \frac{\phi(x)-x^p}{p}$ is the unique $p$-derivation,
where the endomorphism $\phi\colon R\to R$ is the lift of the $p$-th power Frobenius endomorphism of $R/pR$. 

Let
$A$ be an abelian scheme over $R$ of relative dimension $g$. In analogy with differential algebra, one can define
the $n$-th order jet space $J^nA$ to be the $\pi$-adic formal scheme over $R$ with functor of points
$$(J^nA)(C) = A(W_n(C)) = \Hom_R(\Spec W_n(C),A), $$
where $C$ is an $R$-algebra and $W_n(C)$ is the ring of $\pi$-typical Witt 
vectors of length $n+1$, which is taken as the arithmetic analogue of 
$C[t]/(t^{n+1})$. One might say that $W_n(C)$ is the ring of
arithmetic jets of order $n$. The jet space $J^nA$ is also known as the Greenberg transform. Similar to the
differential jet space, $J^nA$ has relative dimension $g(n+1)$ over the base, 
in this case $\Spf R$.

Then we let $\bX_n(A)$ denote the $R$-module of all homomorphisms $J^nA\to\hG$ (of $\pi$-adic formal group
schemes), which are referred to as arithmetic differential characters of $A$ of order $n$. They are the arithmetic
analogue of Manin characters. Let $\bX_\infty(A)$ be the direct limit of the $\bX_n(A)$. The usual Frobenius
operator on Witt vectors induces a canonical Frobenius morphism $\phi:J^{n+1}A \map J^nA$ lying over the
endomorphism $\phi$ of $\Spf R$. Hence pulling back morphisms via $\phi$ as $\Theta \mapsto \phi^*\Theta$, endows
$\bX_\infty(A)$ with an action of $\phi^*$ and hence makes $\bX_\infty(A)$ into a left module over the twisted
polynomial ring $R\{\phi^*\}$ with commutation law $\phi^*\cdot r=\phi(r)\cdot \phi^*$.

In \cite{bui95}, Buium studied the structure of $\bX_\infty(A)_K:= \bX_\infty(A)\otimes_R K$ as a
$K\{\phi^*\}$-module, where $K=R[\frac{1}{p}]$. For example, he showed that when $A$ has ordinary reduction and
its Serre--Tate parameters are either trivial or sufficiently general, $\bX(A)_K$ is freely generated as a
$K\{\phi^*\}$-module by $g$ characters of order either $1$ or $2$. In the case of elliptic curves, he showed
this without any restrictions on the Serre--Tate parameter. 

The main purpose of this paper is not to study the abstract structure of the $\bX_{n}(A)$ but to use them and
other character groups to construct a new filtered $F$-isocrystal associated to $A$. But we will obtain some
integral, $R$-linear refinements of Buium's theorems along the way and as applications at the end.

\vspace{3mm}

To go into greater detail, we need to introduce some further notation.
Let $R$ be a $p$-adic complete
discrete valuation ring with perfect residue field and uniformizer $\pi$.
We always assume $p\mid\pi^{p-2}$, which is to say the absolute ramification index is at most $p-2$. 
Let $\phi:R \map R$ be a
fixed lift of the $p$-th power Frobenius: so $\phi(x) \equiv x^q \mod \pi R$, for all $x \in R$.
(In the body of the paper, we allow $q$-power Frobenius lifts for any power $q$ of $p$, which
requires modifying the Witt vector functors.)
Let us also write $M_K:= M \otimes_R K$ for any $R$-module $M$.


Let $u:J^nA \map A$ be the usual projection map and put $N^n =\ker{u}$. 
Then since $A$ is smooth,
$N^n$ is isomorphic to the ($\pi$-adic formal) affine $ng$-space $\hat{\mathbb{A}}^{ng}_R$.
As a group object, it is unipotent and commutative. 
Then in section~\ref{sec-lat-frob}, we construct morphisms
	$$
	\mfrak{f}:N^{n+1} \map N^n
	$$ 
reducing modulo $\pi$ to the Frobenius map composed
with the usual projection $u:N^{n+1} \map N^n$. We call $\mfrak{f}$ the
{\it lateral Frobenius}. This has a transparent description when $A$ is $\hG$. 
Then
$J^n\hG(C)$ is $W_n(C)$, and $N^n$ is the Veschiebung ideal $VW_n(C)$, and
$\mfrak{f}$ is simply the morphism $V(x)\mapsto VF(x)$. 

As befits its name,
$\mfrak{f}$ is not compatible with the usual
Frobenius $\phi:J^{n+1}A \map J^n A$ under  the inclusions $i:N^n\to J^n A$. 
That is, we generally have
	$$
	\phi \circ i \neq i \circ \mfrak{f}.
	$$
In fact, we can not expect
to have equality here because that would induce lift of Frobenius on 
the quotient $A$, and that exists only when $A$ is a canonical lift. But we do have the equality
	$$
	\phi^2 \circ i = \phi \circ i \circ \mfrak{f}.
	$$ 
We will see that this odd-looking behavior is just a reflection of the familiar
facts $FV\neq VF$ and $FFV=FVF$ in Witt vectors of rings which are $p$-torsion free.

In section 5, we determine the effect of the lateral Frobenius on the character group 
$\Hom(N^\infty,\hG)=\varinjlim_n\Hom(N^n,\hG)$ and show that it is freely generated as an
$R\{\mfrak{f}^*\}$-module by $\Hom(N^1,\hG)$ in the sense that the natural map is an isomorphism:
	$$
	R\{\mfrak{f}^*\}\otimes_R \Hom(N^1,\hG) \longisomap \Hom(N^\infty,\hG).
	$$

In section~\ref{Fiso}, we define $\bH(A)_K$, our eventual isocrystal. We first define
	$$
	\bH(A) 	:= \frac{\Hom(N^\infty,\hG)}{i^*\phi^*(\bX_{\infty}(A)_{\phi})}
	= \varinjlim_n\frac{\Hom(N^n,\hG)}{i^*\phi^*(\bX_{n-1}(A)_{\phi})}
	$$
and show the semi-linear endomorphism $\mfrak{f}^*$ on $\Hom(N^\infty,\hG)$ descends to
$\bH(A)$. It also comes with a submodule
	$$
	\bXp(A):= \varinjlim_n \bX_n(A)/\phi^*(\bX_{n-1}(A)_{\phi})
	$$	
and fits in a map of short exact sequences 
\begin{equation}
	\label{diag-crys-limit-intro}
	\xymatrix{
	0 \ar[r] & \bXp(A)_K \ar[d]_\Upsilon \ar[r] & 
	\bH(A)_K \ar[d]_\Phi \ar[r] &\bI(A)_K \ar@{^{(}->}[d] \ar[r] &  0 \\
	0 \ar[r] & H^0(A,\Omega_A) \ar[r] & \Hdr \ar[r] & H^1(A,\Ou_A) \ar[r] & 0 
	}
\end{equation}

At this point we have not yet shown that $\bH(A)_K$ is finite dimensional.
This is done in section \ref{exactsequence}, where the first step is to prove the following:
\begin{theorem}
\label{maxorder-intro}
For any abelian scheme $A$ of dimension $g$, the character group $\bX_\infty(A)_K$ is freely
generated as a $K\{\phi^*\}$-module by $g$ differential characters of order at most $g+1$.
\end{theorem}

We note that Buium's theorem A in~\cite{bui95} contains the same result but with no bound on 
the orders of the characters. He did however establish exactly the bound $g+1$ for the analogous question
in the setting of differential algebra. See theorem $1.1$ of chapter 5 of his book~\cite{buiHbook}.
(We would like to remark that our 
techniques work in the case of differential algebra 
as well to give an alternate proof.) We emphasize that this bound is not sharp for generic $A$.
Buium's theorem B in~\cite{bui95} says that for abelian varieties with
ordinary reduction and generic Serre--Tate parameters, $\bX_\infty(A)_K$ is freely generated by characters $g$ of order at most $2$. But it would be interesting to determine whether it is the best bound that applies for all $A$.

Our main result then follows:
\begin{theorem}
\label{isocrys-intro}
The rank of $\bH(A)_K$ is at most $2g$, and hence $\bH(A)_K$ is a filtered isocrystal.
The underlying filtered vector space
has a natural map (\ref{diag-crys-limit-intro}) to de Rham cohomology.
\end{theorem}
See theorem~\ref{isocrys} for a more detailed statement.

Next, let us define the {\it lower splitting number} $\mlow$ to be the minimal number which is the order of a
nonzero character: so $\bX_{\mlow}(A) \ne \{0\}$ but $\bX_{\mlow-1}(A) = \{0\}$. We will see that $\mlow$ is in
fact either $1$ or $2$.
If $A$ is an elliptic curve, then $\mlow$ is $1$ when $A$ is the canonical lift of an ordinary elliptic
curve and is $2$ otherwise. In either case, $\bX_{\mlow}(A)$ is a free $R$-module of rank $1$.
We then show in theorem~\ref{phigen-body}
that $\bH(A)$ is finitely generated as an $R$-module. (We do not consider whether this is true for general abelian varieties.) Using this, we prove the following theorem, which is an integral refinement
of a part of Buium's theorem $\text{B}'$ in~\cite{bui95}:

\begin{theorem}
\label{phigen-intro}
Let $A$ be an elliptic curve with lower splitting number $\mlow$. 
Then the $R$-module $\bX_{\mlow}(A)$, which is free of 
rank $1$, freely generates $\bX_\infty(A)$ as an $R\{\phi^*\}$-module
in the sense that the canonical map 
$$R\{\phi^*\}\otimes_R \bX_{\mlow}(A) \map \bX_\infty(A)$$ 
is an isomorphism.
\end{theorem}

In fact, given \'{e}tale coordinates for $A$ at the origin, we construct a 
canonical basis element $\Theta_{\mlow} \in \bX_{\mlow}(A)$. See theorem~\ref{phigen-body} for the proof.

We next consider the finer structure of the isocrystal $\bH(A)_K$, still assuming $A$ is an elliptic curve.
Theorem \ref{elliptic-crystal} describes it explicitly in terms of certain arithmetic-differential modular
parameters $\lambda$ and $\gamma$. We emphasize that these modular parameters are not modular functions in the
usual sense, which is to say functions on the usual moduli space or on bundles over it, but are instead functions
on their arithmetic jet spaces. We show that the parameter $\gamma$ is always divisible by $\pi$ and then
observe that whenever $\gamma/\pi \not\equiv 0 \bmod \pi$, the $F$-crystal $\bH(A)$ is weakly admissible. Using
the Fontaine functor~\cite{fon79} \cite{rap-zink}, one then obtains apparently new crystalline Galois
representations attached to elliptic curves and presumably abelian varieties, or at least to generic ones.

The crystalline theory also attaches a filtered 
$F$-isocrystal $\bH_{\mathrm{crys}}(A)_K$ to $A$. However, our $F$-isocrystal 
$\bH(A)_K$
is different than the crystalline one. This is because the Frobenius 
map on $\bH(A)_K$ depends on the higher $\pi$-derivatives of the coefficients 
of the equations defining the abelian scheme, whereas 
$\bH_{\mathrm{crys}}(A)_K$ does not involve any such higher $\pi$-derivatives.
A natural question is whether the two determine each other, especially by some explicit
linear-algebraic functor like the Fontaine functor mentioned above. 
In the analogous Drinfeld module setting of~\cite{Drin-us},
the shtuka necessarily determines both, simply because it determines the Drinfeld module.
But it would be interesting to go further and describe the functor in purely linear-algebraic terms, without
a detour back through the Drinfeld module.
Even further, one could ask whether the isocrystal is determined by just the local shtuka.
If so, then it would be natural to 
hope that one could do the same in the mixed-characteristic context of this paper
using the $p$-adic shtukas of Scholze and collaborators~\cite{Scholze-icm}.

{\bf Acknowledgement.} We wish to thank the anonymous referee for carefully
reading our
article and the suggestions which led to deeper clarifications and 
enrichment of this paper.

The second author is also grateful to Max Planck Institute for Mathematics in 
Bonn for its hospitality and financial support.

\section{Notation}
\label{notation}
We collect here the notation that will remain fixed throughout the paper.
\begin{align*}
	p &= \text{a prime number} \\
	q &= \text{a power $>1$ of }p \\
	R &= \text{a $p$-adically complete discrete valuation ring} \\
	K &= \text{the fraction field of $R$} \\
	M_K &= K\otimes_R M, \text{ for any $R$-module $M$} \\
	\mfrak{m} &= \text{the maximal ideal of }R\\
	\pi &= \text{a generator of }\mfrak{m} \\
	v &= \text{the valuation on $R$ normalized such that }v(\pi)=1\\
	e &= \text{the absolute ramification index $v(p)$, assumed $\leq p-2$} \\
	k &= \text{the residue field of $R$, assumed to be perfect} \\
	\phi &= \text{an endomorphism of $R$ satisfying $\phi(x) \equiv x^q \bmod \mfrak{m}$, for all $x \in R$} \\
	S &= \Spf (R)
\end{align*}
(The assumption $e \leq p-2$ is only used in the citations to~\cite{bui95} in section~\ref{section:kernel}
and might well be removable.)
By a \emph{$\pi$-formal scheme}, we will mean a $\pi$-adic formal scheme over $S$.

For any $R$-module $M$
if $x_1,\cdots, x_n \in M$ forms an $R$-basis, then we will denote $M =
R\langle x_1,\cdots, x_n\rangle$. If $Y=(y_1,\dots, y_l) \in M^l$ and 
$\Lambda= (\alpha_1,\dots, \alpha_l) \in R^l$, then the dot-product 
$\Lambda.Y$ will denote
$$\Lambda. Y = \lam_1 y_1 + \cdots + \lam_l y_l.$$


\section{Witt Vectors}
Witt vectors over Dedekind domains with finite residue fields were introduced in \cite{bor11a}. We will give a brief over view in this section.

\subsection{Frobenius lifts and $\pi$-derivations}
Let $B$ be an $R$-algebra, and let $C$ be a $B$-algebra with structure map $u:B \map C$.
In this paper, a ring homomorphism $\psi:B\map C$ will be called a
{\it lift of Frobenius} (relative to $u$) if it satisfies the following:
\begin{enumerate}
\item The reduction mod $\pi$ of $\psi$ is the $q$-power Frobenius relative to $u$, that is,
$\psi(x) \equiv u(x)^q \bmod \pi C$.

\item The restriction of $\psi$ to $R$ coincides with the fixed $\phi$ on $R$,
that is, the following diagram commutes
        $$\xymatrix{
        B \ar[r]^\psi & C \\
        R \ar[r]_{\phi} \ar[u] & R \ar[u] 
        }$$
\end{enumerate}
A \emph{$\pi$-derivation} $\d$ from $B$ to $C$ means a set-theoretic map
$\d:B \map C$ satisfying the following for all $x,y \in B$
        \begin{eqnarray*}
        \label{der}
        \d(x+y) &=& \d (x) + \d (y) + C_\pi(u(x),u(y)) \\
        \d(xy) &=& u(x)^q \d (y) +  \d (x) u(y)^q + \pi \d (x) \d (y),
        \end{eqnarray*}
where $C_\pi(X,Y)$ denotes the polynomial
	$$
	C_\pi(X,Y) = \frac{X^q + Y^q - (X+Y)^q}{\pi} \in R[X,Y],
	$$
such that for all $r\in R$, we have
$$
\d(r) = \frac{\phi(r)-r^q}{\pi}.
$$
When $C=B$ and $u$ is the identity map, we will call this simply a $\pi$-derivation on $B$.

It follows that the map $\phi: B \map C$ defined as
        $$
        \phi(x) := u(x)^q + \pi \d (x)
        $$
is a lift of Frobenius in the sense above. Conversely,
for any flat $R$-algebra $B$ with a lift of Frobenius $\phi$, one can define
the $\pi$-derivation $\d(x)= \frac{\phi(x)-x^q}{\pi}$ for all $x \in B$.

It is worth pointing out that although the definition of $\pi$-derivation depends on the choice of uniformizer $\pi$ in a literal sense, it is independent of the choice up to a canonical bijection. Indeed,
if $\pi'$ is another uniformizer, then $\d(x)\pi/\pi'$ is a $\pi'$-derivation, and so this correspondence
induces a bijection between $\pi$-derivations $B\to C$ and $\pi'$-derivations $B\to C$. Further,
if $\pi''$ is a third uniformizer, then we have $(\d(x)\pi/\pi') \pi'/\pi''=\d(x)\pi/\pi''$. In this way,
the set of $\pi$-derivations $\d:B\to C$ is independent of $\pi$, up to a canonical and coherent
family of bijections.

\subsection{Witt vectors}
We will present three different points of view on $\pi$-typical Witt vectors.
Let $B$ be an $R$-algebra with structure map $u:R\to B$.

(1) The ring $W(B)$ of \emph{$\pi$-typical Witt vectors} can be defined as
the unique (up to unique isomorphism) $R$-algebra $W(B)$ with a $\pi$-derivation $\d$ on
$W(B)$ and an $R$-algebra homomorphism $W(B) \map B$ such that, given any
$R$-algebra $C$ with a $\pi$-derivation $\d$ on it and an $R$-algebra map
$f:C \map B$, there exists a unique $R$-algebra homomorphism $g:C \map W(B)$  such that the diagram
        $$
        \xymatrix{
        W(B) \ar[d] & \\
        B & C \ar[l]_f \ar[ul]_g 
        }
        $$
commutes and $g \circ \d = \d \circ g$.
Thus $W$ is the right adjoint of the forgetful functor from $R$-algebras with $\pi$-derivation
to $R$-algebras. For details, see section 1 of~\cite{bor11a} and \cite{joyal}.

(2) If we restrict to flat $R$-algebras $B$, then we can ignore the concept of $\pi$-derivation
and define $W(B)$ simply by expressing the universal property above
in terms of Frobenius lifts, as follows:
Given a flat $R$-algebra $B$, the ring $W(B)$ is
the unique (up to unique isomorphism) flat $R$-algebra $W(B)$ with a lift of Frobenius (in the sense
above) $F:W(B) \map W(B)$ and an $R$-algebra homomorphism $W(B) \map B$ such that
for any flat $R$-algebra $C$ with a lift of Frobenius $\phi$ (compatible with the one on $R$)
on it and an $R$-algebra map $f:C \map B$,
there exists a unique $R$-algebra homomorphism $g:C \map W(B)$ such that the diagram
        $$
        \xymatrix{
        W(B) \ar[d] & \\
        B & C \ar[l]_f \ar[ul]_g 
        }
        $$
commutes and $g \circ \phi = F \circ g$.

(3) Finally, one can also define Witt vectors in terms of the Witt polynomials.
For each $n \geq 0$, let us define $B^{\phi^n}$ to be the $R$-algebra
with structure map $R \stk{\phi^n} {\map} R \stk{u}{\map} B$ and define the \emph{ghost rings}
to be the product $R$-algebras
$\Pi^n_{\phi} B = B \times B^{\phi} \times \cdots  \times B^{\phi^n}$ and
$\Pi_{\phi}^\infty B= B \times B^{\phi} \times \cdots$.
Then for all $n \geq 1$ there exists a \emph{restriction}, or \emph{truncation},
map $T_w:\Pi_{\phi}^nB \map \Pi_{\phi}^{n-1}B$ given by $T_w(w_0,\cdots,w_n)= (w_0,\cdots,w_{n-1})$.
We also have the left shift \emph{Frobenius} operators $F_w:\Pi_{\phi}^n B \map \Pi_{\phi}^{n-1} B$ given by
$F_w(w_0,\dots,w_n) = (w_1,\dots,w_n)$. Note that $T_w$ is an $R$-algebra morphism, but
$F_w$ lies over the Frobenius endomorphism $\phi$ of $R$.

Now as sets define
        \begin{equation}
        \label{eq-witt-coord}
        W_n(B)=B^{n+1}, 
        \end{equation}
and define the set map $w:W_n(B) \map \Pi_{\phi}^n B$  by $w(x_0,\dots,x_n)= (w_0,\dots,w_n)$ where
        \begin{equation}
        \label{eq-witt-poly}
        w_i = x_0^{q^i}+ \pi x_1^{q^{i-1}}+ \cdots + \pi^i x_i
        \end{equation}
are the \emph{Witt polynomials}.
The map $w$ is known as the {\it ghost} map. (Do note that under the traditional indexing our $W_n$ would be
denoted $W_{n+1}$.) We can then define the ring $W_n(B)$, the ring
of truncated $\pi$-typical Witt vectors, by the following theorem as for example in 
\cite{hessl05}, page 141:

\begin{theorem}
\label{wittdef}
For each $n \geq 0$, there exists a unique functorial $R$-algebra structure on $W_n(B)$ such that
$w$ becomes a natural transformation of functors of $R$-algebras.
\end{theorem}

The agreement between descriptions (1)--(3) above can be seen as follows: (2) is a particular case of (1), and
conversely, (1) is determined by (2) and functoriality because every $R$-algebra is a quotient of a flat
$R$-algebra. To show (3) and (1) are equivalent, it is again enough to restrict to flat $R$-algebras, in which
case it follows from Lazard~\cite{Lazard:CFG}, VII\S 4, or Bourbaki~\cite{Bourbaki:CommAlg}, IX.44, exercise 14a.
Alternatively, one can see sections 1--3 of~\cite{bor11a}.

The realization that Witt vectors, which have traditionally been defined as in (3) above, have a satisfying
definition as the solution to a universal property, as in (1), is due to Joyal~\cite{joyal}.

\subsection{Operations on Witt vectors}
\label{subsec-witt-operations}
Now we recall some important operators on the Witt vectors. 
They are the unique functorial operators corresponding under the ghost map
to the operators $T_w$, $V_w$, and $F_w$ on the ghost rings defined above.
First, the \emph{restriction}, or \emph{truncation}, maps $T:W_n(B) \map W_{n-1}(B)$ 
are given by $T(x_0,\dots,x_n) = 
(x_0,\dots, x_{n-1})$. Note that $W(B) = \varprojlim W_n(B)$.
There is also the {\it Frobenius} ring homomorphism
$F:W_n(B) \map W_{n-1}(B)$, which can be described in terms of the ghost map.
It is the unique map which is functorial in $B$ and makes the
following diagram commutative
        \begin{equation}
        \xymatrix{
        W_n(B) \ar[r]^w \ar[d]_F & \Pi^n_{\phi} B \ar[d]^{F_w} \\
        W_{n-1}(B) \ar[r]_-w & \Pi_{\phi}^{n-1} B^n
        } \label{F}
        \end{equation}
As with the ghost components, $T$ is an $R$-algebra map but $F$ lies over the Frobenius endomorphism $\phi$
of $R$.

Next we have the {\it Verschiebung} $V:W_{n-1}(B) \map W_n(B)$ given by
$$V(x_0,\dots,x_{n-1}) = (0,x_0,\dots,x_{n-1}).$$ Let
$V_w:\Pi_{\phi}^{n-1}B \map \Pi_{\phi}^n B$ be the additive map given by
$$V_w(w_0,..,w_{n-1})= (0, \pi w_0,\dots,\pi w_{n-1}).$$ Then the Verschiebung
$V$ makes the following diagram commute:
        \begin{equation}
        \xymatrix{
        W_{n-1}(B) \ar[r]^-w \ar[d]_V & \Pi_{\phi}^{n-1}B \ar[d]^{V_w} \\
        W_n(B) \ar[r]_-w & \Pi_{\phi}^nB
        }\label{V}
        \end{equation}
For all $n \geq 0$ the Frobenius and the Verschiebung satisfy the identity
        \begin{equation}
        \label{FV-pi}
        FV(x) = \pi x.
        \end{equation}
The Verschiebung is not a ring homomorphism (unless $B=0$), but it is $\bZ$-linear.

Finally, we have the multiplicative Teichm\"uller map $[~]:B \map W_n(B)$ given by
$x\mapsto [x]= (x,0,0,\dots)$.

\subsection{Prolongation sequences and jet spaces}
Let $X$ and $Y$ be $\pi$-formal schemes over $S=\Spf R$. We say a pair 
$(u,\d)$ is a {\it prolongation}, and write 
$Y \stk{(u,\d)}{\map} X$, if $u: Y \map X$ is a map of $\pi$-formal schemes 
over $S$ and $\d: \Ou_X \map u_*\Ou_Y$ is a 
$\pi$-derivation making the following diagram commute: 
	$$
	\xymatrix{
	R \ar[r] &  u_* \Ou_Y \\
	R \ar[u]^\d \ar[r] &  \Ou_X \ar[u]_\d \\
	} 
	$$ 
Following Buium~\cite{bui00}(page 103), a {\it prolongation sequence} is a sequence 
of prolongations
	$$
	\xymatrix{
	S & T^0 \ar_-{(u,\d)}[l] & T^1 \ar_-{(u,\d)}[l] & \cdots\ar_-{(u,\d)}[l]},
	$$
where each $T^n$ is a $\pi$-formal scheme over $S$, satisfying 
$$u^* \circ \d = \d \circ u^*
$$
and $u^*$ is the pull-back morphism on the sheaves induced by $u$.
We will often use the 
notation $T^*$ or $\{T_n\}_{n \geq 0}$.
Note that if the  $T^n$ are flat over $S$ then having a 
$\pi$-derivation $\d$ is equivalent to having lifts of Frobenius $\phi:T^{n+1}\to T^n$.

Prolongation sequences form a category $\mcal{C}_{S^*}$, where a morphism $f:T^*\to U^*$ is 
a family of morphisms $f^n:T^n\to U^n$ commuting with both the $u$ and $\d$, in the evident sense.
This category has a final object $S^*$ given by $S^n=\Spf R$ for all $n$, where each $u$ is the identity and
each $\d$ is the given $\pi$-derivation on $R$.


For any $\pi$-formal scheme $X$ over $S$, for all $n \geq 0$ we define the 
$n$-th jet space $J^nX$ (relative to $S$) as a functor by
	$$
	J^nX (C) :=  X(W_n(C)) = \Hom_S(\Spf(W_n(C)),X),
	$$
for any $R$-algebra $C$.
This functor is representable by a $\pi$-formal scheme.
(This was established in two independent threads. First, one can use theorem 12.1 of~\cite{bor11b} to prove the 
representability of the absolute arithmetic jet space $W_{n*}(X)$ and then invoke the equality 
$J^nX = S\times_{W_{n*}(S)} W_{n*}(X)$. Second, one can use
theorem 2.3.24 of~\cite{arnab-thesis}, which proves Buium's original
construction represents the desired functor.)

Then $J^*X:= \{J^nX \}_{n \geq 0}$ forms a prolongation sequence and is 
called the {\it canonical prolongation sequence} as in proposition 1.1 in
\cite{bui00}. By the same proposition 1.1 in 
\cite{bui00}, $J^*X$ satisfies the following 
universal property---for any $T^* \in \mcal{C}_{S^*}$ and $X$ a $\pi$-formal
scheme over 
$S^0$, we have
\begin{equation}	
\label{canprouniv}
	\Hom(T^0,X) = \Hom_{\mcal{C}_{S^*}}(T^*, J^*X)
\end{equation}

Let $X$ be a $\pi$-formal scheme over 
$S= \Spf R$. Define $X^{\phi^n}$ by $X^{\phi^n}(B) := X(B^{\phi^n})$ for any $R$-algebra $B$. 
In other words, $X^{\phi^n}$ is $X \times_{S,\phi^n} S$, the pull-back of $X$ under the map $\phi^n:S\to S$.
(N.B., $(\Spf A)^{\phi^n}$ should not be confused with $\Spf(A^{\phi^n})$.)
Next define the following product of $\pi$-formal schemes
	$$
	\Pi^n_{\phi} X= X \times_S X^{\phi} \times_S \cdots \times_S X^{\phi^n}.
	$$ 
Then for any $R$-algebra $B$ we have $X(\Pi_{\phi}^n B) = X(B)\times_S \cdots \times_S X^{\phi^n}(B)$.
Thus the ghost map $w$ in theorem \ref{wittdef} defines a map of $\pi$-formal
$S$-schemes 
	$$
	w:J^nX \map \Pi_{\phi}^nX.
	$$ 
Note that $w$ is injective when evaluated on points with coordinates in any flat $R$-algebra.

The operators $F$ and $F_w$ in (\ref{F}) induce maps $\phi$ and $\phi_w$ fitting into a commutative diagram
\begin{equation}
	\label{phiw}
	\xymatrix{
	J^nX \ar[r]^w\ar[d]_{\phi} & \Pi_{\phi}^n X \ar[d]^{\phi_w}\\
	J^{n-1}X \ar[r]_w & \Pi_{\phi}^{n-1} X.
	}
\end{equation}
The map $\phi_w$ is easier to define. It is the left-shift operator given by
	$$
	\phi_w(w_0,\dots,w_n)= (\phi_S(w_1),\dots,\phi_S(w_n)),
	$$
where  $\phi_S:X^{\phi^i} \map X^{\phi^{i-1}}$ is the composition given in the following diagram:
\begin{equation}
	\label{ko}
	\xymatrix{
	X^{\phi^i}\ar[r]^-{\sim} & X^{\phi^{i-1}} \times_{S,\phi} S \ar[d] 
	\ar[r]^-{} & 
	X^{\phi^{i-1}} \ar[d] \\
	& S \ar[r]_{\phi} & S.
	}
\end{equation}
We note that a choice of a coordinate system on $X$ over $S$ induces coordinate systems on 
$X^{\phi^i}$ for each $i$, and with respect to these coordinate systems,
$\phi_S$ is expressed as the identity. One might say that $\phi_S$ 
applies $\phi$ to the horizontal coordinates and does nothing to the vertical coordinates.

For the map $\phi:J^nX\to J^{n-1}X$, we can define it in terms of the functor of points.
For any $R$-algebra $B$, the ring map $F:W_n(B)\to W_{n-1}(B)$ is not $R$-linear
but lies over $\phi:R\to R$. As $B$ varies, the resulting linearized $R$-algebra maps 
	$$
	W_n(B)\to W_{n-1}(B)^{\phi} = W_{n-1}(B^{\phi}),
	$$
induce functorial maps
\begin{equation}
	J^nX(B) = X(W_n(B)) \longmap X(W_{n-1}(B^{\phi})) = J^{n-1}X(B^{\phi}),
\end{equation}
which is the same as giving a morphism $\phi:J^n X\to J^{n-1}X$ lying over $\phi:S\to S$.

If $A$ is a $\pi$-formal group scheme over $S$,
the ghost map $w:J^nA\to \Pi_{\phi}^n A$ and the truncation map 
$u:J^nA\to J^{n-1}A$
are $\pi$-formal group scheme homomorphisms over $S$. 
On the other hand, the Frobenius maps $\phi:J^nA\to J^{n-1}A$ and $\phi_w:\Pi_{\phi}^nA \to \Pi_{\phi}^{n-1}A$
are $\pi$-formal group scheme homomorphisms lying over 
the Frobenius endomorphism $\phi$ of $S$.

\subsection{Character groups}
\label{subsec-char-groups}
Given a prolongation sequence $T^*$ we can define its shift $T^{*+n}$ by 
$(T^{*+n})^j:= T^{n+j}$ for all $j$, page 106 in \cite{bui00}.
	$$
	S \stk{(u,\d)}{\leftarrow} T^n \stk{(u,\d)}{\leftarrow} T^{n+1}\dots 
	$$
We define a {\it $\d$-morphism of order $n$} from $X$ to $Y$ to be a 
morphism $J^{*+n}X \map J^*Y$ of prolongation sequences.
We define a {\it character of order $n$}, $\Theta:A \map \hG$ 
to be a $\d$-morphism of order $n$ from $A$ to $\hG$
which is also a group homomorphism of $\pi$-formal group schemes.
By the universal property of jet schemes as in (\ref{canprouniv}),
 an order $n$ character is equivalent to a homomorphism
$\Theta:J^nA \map \hG$ of $\pi$-formal group schemes over $S$. 
We denote the group of 
characters of order $n$ by $\bX_n(A)$:
	$$
	\bX_n(A)=\Hom_S(J^nA,\hG).
	$$
Note that $\bX_n(A)$ comes with an
$R$-module structure since $\hG$ is an $R$-module object in the 
category of $\pi$-formal schemes over $S$. Also the inverse system 
$J^{n+1}A \stk{u}{\map} J^nA$ defines a directed system 
	$$
	\bX_n(A) \stk{u^*}{\map} \bX_{n+1}(A) \stk{u^*}{\map}\cdots
	$$
via pull back. Each morphism $u^*$ is injective and 
we then define $\bX_\infty(A)$ to be the direct limit $\varinjlim \bX_n(A)$.

Similarly, pre-composing with the Frobenius map $\phi:J^{n+1}A\to J^nA$ induces a Frobenius operator
$\phi:\bX_n(A)\to \bX_{n+1}(A)$. However since $\phi:J^{n+1}A\to J^nA$ is not a morphism over $S$ but
instead lies over the Frobenius endomorphism $\phi$ of $S$, 
some care is required.
Consider the relative Frobenius morphism $\phi_R$, defined to be the unique
morphism making the following diagram commute:
	$$
	\xymatrix{
	J^{n+1}A \ar@{.>}^{\phi_R}[rd] \ar@/^/[rrd]^{\phi} \ar@/_/[ddr]& & \\
	& J^nA \times_{S,\phi} S \ar[d] \ar[r] & J^nA \ar[d] \\
	& S \ar[r]_{\phi} & S 
	}
	$$
Then $\phi_R$ is a morphism of $\pi$-formal  group schemes over $S$.
Now given a $\d$-character $\Theta:J^nA\to \hG$, define $\phi^*\Theta$ to be 
the composition
\begin{equation}
	J^{n+1}A \longlabelmap{\phi_R} J^nA \times_{S,\phi} S 
	\longlabelmap{\Theta\times\mathrm{id}} \hG\times_{S,\phi} S
\longlabelmap{\iota} \hG
\end{equation}
where $\iota$ is the isomorphism of $\pi$-formal group schemes over $S$
coming from the fact that $\hG$ descends to $\bZ_p$ as a $\pi$-formal group 
scheme. For any $R$-algebra $B$, the induced morphism on $B$-points is
	$$
	A(W_{n+1}(B)) \longlabelmap{A(F)} A(W_n(B)^{\phi}) 
\longlabelmap{\Theta_B^{\phi}} 
	B^{\phi} \longlabelmap{b\mapsto b} B.
	$$
Note that this composition is indeed a morphism of $\pi$-formal group schemes.

Thus we have an additive map $\bX_n(A) \to \bX_{n+1}(A)$ given by 
$\Theta\mapsto \phi^*\Theta$. Note
that this map is not $R$-linear. However, the map
	$$
	\phi^*:\bX_n(A) \longmap \bX_{n+1}(A)^{\phi}, \quad \Theta\mapsto \phi^*\Theta 
	$$ 
is $R$-linear, where $\bX_{n+1}(A)^{\phi}$ denotes the abelian group 
$\bX_{n+1}(A)$ with 
$R$-module structure defined by the law $r\cdot \Theta := \phi(r)\Theta$.
Taking direct limits in $n$, we obtain an $R$-linear map
\begin{equation}
\label{phi-pull-back}
	\bX_\infty(A) \longmap \bX_\infty(A)^{\phi}, \quad \Theta\mapsto\phi^*
\Theta.  
\end{equation}
In this way, $\bX_\infty(A)$ is a left module over the twisted polynomial ring $R\{\phi^*\}$ with
commutation law $\phi^*r = \phi(r)\phi^*$.

\section{Lateral Frobenius}
\label{sec-lat-frob}
Let $B$ be an $R$-algebra and  $f:R \map B$ be the structure map. 
Consider the following map of exact sequences:
        $$
        \xymatrix{
        0 \ar[r] & W_{n-1}(B) \ar^V[r] & W_n(B) \ar[r] & B \ar[r] & 0 \\
        0 \ar[r] & W_{n-1}(B) \ar^-V[r]\ar@{=}[u] & R\times_B W_n(B) \ar[r]\ar^I[u] & R \ar[r]\ar^f[u] & 0		
        }
        $$
Here the top row is the usual exact sequence, and the bottom row is the pull back of the top row along 
the structure map $f$. Let us write
	$$
	\Wp_n(B) = R\times_B W_n(B),
	$$
which is naturally an $R$-algebra.
One can think of elements of $\Wp_n(B)$ as usual Witt vectors but
with the initial component being an element of $R$ rather than $B$. Indeed, we have a bijection
\begin{align*}
	R \times B^n &\to R\times_B W_n(B) \\
	(r,b_1,\dots,b_n)&\mapsto \big(r,(f(r),b_1,b_2,\dots,b_n)\big)	
\end{align*}
The addition and multiplication laws on $\Wp_n(B)$ are then defined in terms of these components
by the usual universal polynomials for Witt vector addition and multiplication, keeping in mind that
$B$ is an $R$-algebra.

Now we will define an $R$-algebra homomorphism 
\begin{equation}
\label{map:lateral-Frob}
	\Fp:\Wp_n(B) \map \Wp_{n-1}(B^{\phi}),
\end{equation}
which we call the {\it lateral Frobenius} map.
First observe that the projection $\Wp_n(B)\to R$ has a canonical section. Indeed, the Frobenius lift $\phi$
on $R$ induces a lift 
$$\tilde{f}:R\to W_n(B)$$ 
of $f$ and hence a section of the map $\Wp_n(B)\to R$. Therefore, we have a bijection
\begin{align*}
	R\oplus W_{n-1}(B) &\to \Wp_n(B) \\
	(r,z) &\mapsto \tilde{f}(r)+V(z).
\end{align*}
In terms of this
identification, we then define the lateral Frobenius map (\ref{map:lateral-Frob})
simply to be the Frobenius map on the second summand:
	$$
	(r,z)\mapsto (r,F(z)).
	$$
So $\Fp$ is defined by the expression
	$$
	\Fp(f(r)+V(z)) = f(r) + VF(z).
	$$

There are ghost analogues of $\Wp_n$ and $\Fp$.
Let us define
	$$
	\Pp_{\phi}^n(B)=R\times_B \Pi_{\phi}^n(B) = R\times B^{\phi}\times\cdots\times B^{\phi^n},
	$$
with the product $R$-algebra structure, and the $R$-algebra map 
\begin{align*}
	\Fp_w:\Pp_{\phi}^n(B) &\to \Pp_{\phi}^{n-1}(B^{\phi}) \\
	(r,z_1,z_2,\dots, z_n) &\mapsto (r, z_2,z_3, \dots, z_n),
\end{align*}
which drops the $z_1$ component.
Then the following diagram commutes:
\begin{equation}
\label{diag:lateral-ghost}
	\xymatrix{
		\Wp_n(B) \ar[d]_{\Fp} \ar[r]^w & \Pp_{\phi}^n(B) \ar[d]^{\Fp_w} \\
		\Wp_{n-1}(B^{\phi}) \ar[r]^w & \Pp_{\phi}^{n-1}(B^{\phi})
	}
\end{equation}
Indeed, this follows straight from the description of $V$ in terms the ghost components.
The ghost principle then
implies that $\Fp$ is also an $R$-algebra homomorphism:
Since any $R$-algebra is a quotient of a $\pi$-torsion free one, and
since $\Wp$ preserves surjectivity and $\Fp_w$ is functorial in $B$,
we may assume that $B$ itself is $\pi$-torsion free. Then
the ghost map $w$ is injective and hence it is sufficient, by the diagram above,
to recall that $\Fp_w$ is an $R$-algebra homomorphism.

\begin{proposition}
\label{latfrobw}
Write $\Fp(r,x_1,\dots,x_n) = (r, \Fp_1, \dots, \Fp_{n-1})$.
Then 
	$
	\Fp_i \equiv x_i^q \bmod \pi.
	$
\end{proposition}
\begin{proof}
This is equivalent to showing that the following diagram commutes:
	$$
	\xymatrix{
	W_n(B) \ar^{\Fp}[r]\ar^{V}[d] & W_{n-1}(B^{\phi}) \ar^{V}[d] \\
	W_{n+1}(B/\pi B) \ar^{F}[r] & W_n(B^{\phi}/\pi B^{\phi})
	}
	$$
This follows
from the identity $FV=VF$ on $W_n(B/\pi B)$, which holds since $B/\pi B$ is in characteristic $p$.
\end{proof}

\begin{proposition}
\label{fdid}
Let $I: \Wp_n(B) \map W_n(B)$ be the natural map. Then we have
	$$
	F^2 \circ I = F \circ I \circ \Fp.
	$$
\end{proposition}
\begin{proof}
Again by the ghost principle, it is enough to check the corresponding identity on the ghost vectors.
So, if $I_w: \Pp_{\phi}^n(B) \map \Pi_{\phi}^n (B)$ denotes
the natural inclusion, then it is sufficient to show 
	$$
	F^2_w \circ I_w = F_w \circ I_w \circ \Fp_w.
	$$

This is a straightforward verification. We have 
\begin{eqnarray*}
(F_w \circ I_w \circ \Fp_w) (r,z_1,\dots, z_n) &=& (F_w \circ I_w) (r,z_2,\dots, z_n) \\
&=&F_w(r,z_2,\dots, z_n) \\
&=& (z_2,\dots, z_n).
\end{eqnarray*}
On the other hand,
\begin{eqnarray*}
(F_w^2 \circ I_w)(r,z_1,\dots, z_n) &=& F_w^2(r,z_1,\dots, z_n) \\
&=& F_w(z_1,\dots,z_n) \\
&=& (z_2,\dots,z_n).
\end{eqnarray*}
And this completes the proof.
\end{proof}

{\bf Remark.} Here we would like to note that $F \circ I  \ne I \circ \Fp$. This
is because if we again look at the identity on the ghost side, then
$F_w \circ I_w (z_0,\dots, z_n) = (z_1,\dots, z_n)$, whereas $I_w \circ 
\Fp_w(z_0,\dots, z_n) = (z_0,z_2,\dots, z_n)$ making the identity 
not possible. One can also view this as a manifestation of the inequality $FV\neq VF$.

For any $\pi$-formal $S$-scheme $X$, with an $S$-point $P:S \map X$, for each 
$n$ we can define $N^n = J^nX \times_X S$, which is the following fiber product
$$
\xymatrix{
	N^n \ar[d]  \ar[r]^-i & J^nX \ar[d] \\
	S \ar[r]^P  & X 
}$$
Thus $N^n(B)$ is the set of thickenings of the point $P$ to a map $h:\Spf (\Wp_n(B))\to X$:
	$$
	\xymatrix{
	\Spf(R) \ar^{P}[r]\ar[d] & X \\
	\Spf (\Wp_n(B)) \ar@{-->}^{h}[ur]
	}
	$$

The projections $\Wp_n(B)\map \Wp_{n-1}(B)$ on Witt vectors therefore induce
maps $u:N^n \map N^{n-1}$.
Similarly, the lateral Frobenius maps $\Fp:\Wp_n(B) \map \Wp_{n-1}(B^{\phi})$ on Witt vectors induce 
lateral Frobenius maps in geometry
	$$
	\mff: N^n \map N^{n-1}.
	$$
By definition, they send a morphism $h$ as above to its composition with the Witt vector lateral Frobenius
	$$
	\Spf (\Wp_{n-1}(B^{\phi})) \longlabelmap{\Fp} \Spf (\Wp_n(B)) \longlabelmap{h} X.
	$$
Indeed, $\Fp$ is a morphism of $R$-algebras compatible with the projections to $R$.
The lateral Frobenius $\mff:N^n\map N^{n-1}$ is a group homomorphism, but 
we emphasize however that it only is semi-linear. It
does not lie over the identity of $S$ but over the 
endomorphism $\phi$ of $S$. This is because it is the twist $B^{\phi}$,  and not $B$,
which appears in $\Wp_{n-1}(B^{\phi})$ in the definition of $\mff$.

\begin{theorem}
\label{latfrob}
The morphism $\mff: N^n \map N^{n-1}$ is a lift of Frobenius and satisfies
	$$ 
	\phi^{\circ j} \circ i = \phi^{j-1} \circ i \circ \mff
	$$
for $j\geq 2$. If $X$ is smooth over $S$, then $\{N^n\}_{n=1}^\infty$ forms a prolongation sequence.
\end{theorem}
{\it Proof.} 
The statement that $\mff$ is a lift of Frobenius with respect to $u$ is exactly proposition \ref{latfrobw}.
For the compositional identity, we may assume $j=2$:
$$
\phi^{\circ 2} \circ i = \phi \circ i \circ \mff,
$$
which is an immediate consequence of proposition \ref{fdid}.

If $X$ is smooth, then the $N^n$ are smooth and hence flat over $S$ for all $n$. 
Therefore the lift of 
Frobenius $\mff:N^n \map N^{n-1}$ corresponds to a $\pi$-derivation $\d$ with 
respect to $u$ on
the corresponding structure sheaves hence making the system 
$\{N^n\}_{n=1}^\infty$ into a prolongation sequence.  $\qed$

Let $A$ be a $\pi$-formal group scheme over $S$ with identity section $e:S 
\map A$. Let us define $\Ppp{n} A = S \times_A \Pi_{\phi}^n A $ as the 
following fiber product
$$\xymatrix{
\Ppp{n}A \ar[r]^{i_w} \ar[d] & \Pi_{\phi}^n A \ar[d]^{u_w}\\
S \ar[r]_e & A
}$$
where $u_w$ is the usual projection onto the initial coordinate 
of $\Pi_{\phi}^n A$. Let us also 
fix the notation $i_w: \Ppp{n}A \map \Pi_{\phi}^n A$ for the induced morphism, as shown.
Then a point $w \in \Ppp{n}A$ can be expressed as 
$z= (e,z_1,\dots, z_n)$ where $z_i \in A^{\phi^i}$.

Then there are associated morphisms $\up_w, \mfrak{f}_w: \Ppp{n}A
\map \Ppp{n-1}A$, corresponding to $u$ and $\mfrak{f}$ respectively
and they are given by 
\begin{align}
\label{phimore}
\up_w(e,z_1,\dots, z_n) &= (e,z_1,\dots, z_{n-1})  \nonumber \\ 
\mfrak{f}_w(e,z_1,\dots, z_n) &= (e,\phi_S(z_2),\dots, \phi_S(z_n)) 
\end{align}

We then have the following morphism of short exact sequences of $\pi$-formal
 group schemes 
$$
\xymatrix{
0 \ar[r] & N^n \ar[r]^i \ar[d]_-w & J^nA \ar[r]^u \ar[d]_-w & A \ar[r]
 \ar[d] & 0 \\
0 \ar[r] & \Ppp{n}A \ar[r]^{i_w} & \Pi_{\phi}^n A \ar[r]^{u_w}
 & A \ar[r] & 0
} $$

\subsection{Local Witt coordinates}
\label{localWitt}
Assume further that $A$ is a smooth commutative $\pi$-formal group scheme over 
$S$ of relative dimension $g$. 
We now recall a few facts on the \'{e}tale coordinates to compute the 
characters out of $N^n$ to $\hG$, explicitly in coordinates. More details
can be found in, say \cite{bui95}, page 317. 

Recall if $A \map Y$ is an \'{e}tale morphism of $\pi$-formal schemes, then
$J^nA \simeq J^nY \times_Y A$. 
Now assume further that $Y = \Spf~R[\bx_0]\h$, where $\bx_0=(x_{01},\cdots, x_{0g})$ is a
system of \'{e}tale coordinates at the identity section of $A$. We claim that 
$J^nY \simeq W_n^g$, where $W_n$ is
the $\pi$-formal scheme representing the Witt vector of length $n+1$ functor.
Recall, by the definition of the $n$-th arithmetic jet functor, for any $R$-algebra $B$
 we have 
\begin{equation}
J^nY(B) = \Hom_R(R[\bx_0]\h,W_n(B))
\end{equation}
Now giving an element in $\Hom_R(R[\bx_0]\h,W_n(B))$ is equivalent to specifying
the image of the tuple of variables $\bx_0$ in $W_n(B)$. Hence the set
$\Hom(R[\bx_0]\h,W_n(B))$ is isomorphic to $W_n^g$. Therefore $J^nY \simeq 
\Spf R[\bx_0,\dots, \bx_n]\h$, where $\bx_i=(x_{i1},\cdots, x_{ig})$ is the 
Witt coordinate. Therefore we have
\begin{equation}
\label{jcoord}
J^nA \simeq A \times_Y J^nY  \simeq A \times_Y \Spf~R[\bx_0,\dots, \bx_n]\h.
\end{equation}
Hence the $\bx_i$ form an \'{e}tale coordinate system on $J^nA$ around the
identity section. We will call them the local Witt coordinates for $J^nA$ 
since they are induced from the Witt vectors as seen above.
This, in turn, induces a coordinate system on $\Lie J^nA$. 
Also recall that  
$N^n = J^nA \times_A S$, where the fiber-product is taken along the identity 
section $S \stk{e}{\map} A$. Then by composition, we have a section $S 
\stk{e}{\map} A \map Y$ which, at the level of maps between $R$-algebras 
$R[\bx_0] \map R$, is given by $\bx_0 \map 0$. Then we have 
\begin{align*}
N^n &= J^nA \times_A S = (J^nY \times_Y A) \times_A S = J^nY \times_Y S \\
&\simeq \Spf R[\bx_1,\dots, \bx_n]\h 
\end{align*}
For the rest of the paper, we will be using the coordinate system defined
by $\bx_i$ on $N^n$. We emphasize however that once $n>1$, it is not the same coordinate system
as the one used
by Buium, as in corollary (1.5) of~\cite{bui95}, for example. We call our coordinates the
\emph{Witt coordinates} to distinguish them from what might be called the Buium--Joyal coordinates
used in~\cite{bui95}.

\begin{theorem}
\label{comfact}
Let $A$ be a smooth $\pi$-formal group scheme over $S$. Then
the morphism of $\pi$-formal group schemes 
$(i \circ \mfrak{f} - \phi \circ i)\colon N^n \map 
J^{n-1}A$ factors uniquely through  $u:N^n\to N^1$:
$$\xymatrix{
N^n \ar[rr]^-{i \circ \mfrak{f} -\phi \circ i}  \ar[d]_u & & J^{n-1}A \\
N^1 \ar@{..>}[rru]_g & &
}$$
\end{theorem}
{\it Proof.} 
Let $e: S \map A$ denote the identity section. Consider the following diagram
$$\xymatrix{
N^n \ar[r]^w \ar[d]_{x \mapsto (x,-x)} & \Ppp{n}
 A \ar[d]^{z \mapsto (z,-z)}  \\
N^n \times N^n \ar[d]_{i \times \mfrak{f}} 
\ar[r]^w  & \Ppp{n} A \times \Ppp{n} A \ar[d]^{i_w \times 
\mfrak{f}_w}\\
J^n A \times N^{n-1} \ar[d]_{\phi \times i} \ar[r]^w & 
\Pi_{\phi}^nA \times \Ppp{n-1}A \ar[d]^{\phi_w \times i_w} \\
J^{n-1}A \times J^{n-1}A \ar[d]_{(x_1,x_2) \mapsto x_1+ x_2} \ar[r]^w & 
\Pi_{\phi}^{n-1}A \times \Pi_{\phi}^{n-1}A \ar[d]^{(z_1,z_2) \mapsto z_1+z_2}  \\
J^{n-1}A \ar[r]^w & \Pi_{\phi}^{n-1}A 
}$$

Then the composition of the left column is precisely 
$i \circ \mfrak{f} - \phi \circ i$, whereas the right column sends 
$z= (e,z_1, \dots, z_n) \mapsto (\phi_S(z_1), e, \dots, e)$. (This follows from the
definitions as in (\ref{phiw}) and (\ref{phimore}).) Hence the map $\Ppp{n} A \map \Pi_{\phi}^{n-1}A$
factors as $\Ppp{n} A \map \Ppp{1} A \stk{g_w}{\longrightarrow} \Pi_{\phi}^{n-1}A$ 
given by 
	$$
	(e,z_1, \cdots, z_n) \mapsto (e,z_1) \mapsto (\phi_S(z_1),e, \cdots, e).
	$$
Hence we have the solid arrows in the following commutative diagram:
	$$
	\xymatrix{
	& N^1 \ar@{.>}[ddl]|\hole^(.25){g} \ar[r]^w & \Ppp{1} A 
	\ar[ddl]^{g_w} \\
	N^n \ar[ru]^u \ar[d]_{i \circ \mfrak{f}- \phi \circ i} \ar[r]^w & 
	\Ppp{n} A \ar[ru]  \ar[d]\\
	J^{n-1}A \ar[r]^w & \Pi_{\phi}^{n-1} A
	}$$
Now we claim that there is a unique map $g:N^1 \map J^{n-1}A$, as shown above, making the 
entire diagram commute.
Observe that with the choice of local Witt coordinates of $J^nA$ as in 
(\ref{jcoord}) the map $u:N^n\to N^1$ admits a section $\sigma$ which is a
morphism of $\pi$-formal schemes.
Hence there is at most one map $g$ making the diagram commute.

We can also use the section $\sigma$ to prove existence.
Put $g=(i \circ \mfrak{f}- \phi \circ i)\circ \sigma$. 
It is then enough to prove the commutativity relation
$i \circ \mfrak{f}- \phi \circ i = g \circ u$. By
the diagram above, these statements will follow from the injectivity of the map
        $$
        (J^{n-1}A)(B) \to (\Pi_{\phi}^{n-1} A)(B)
        $$
when $\Spf B = N^n$. By adjointness, this is equivalent to the
injectivity of the map $A(W_{n-1}(B))\to A(\Pi_{\phi}^{n-1}B)$. To show this, it is enough to show that $\Spf
\Pi_{\phi}^{n-1}B \to \Spf W_{n-1}(B)$ is an epimorphism in the category of 
$\pi$-formal schemes. To show this, it is enough to
show that the ghost map $W_{n-1}(B)\to \Pi_{\phi}^{n-1}B$ is injective. (See 
\cite{ega1} (9.5.6).) But this holds
because $N^n$ is smooth over $S$ and hence $B$ is flat.
$\qed$

As in the remark above, this theorem can be understood as a manifestation of the fact that the operator
$FV-VF$ on Witt vectors can be expressed as a function of just the initial component.

\section{Characters of the kernel} 
\label{section:kernel}
We continue with our notations as defined in section \ref{notation}. Additionally, we will also
fix the following for the rest of the paper:
\begin{align*}
	A &= \text{a $\pi$-formal abelian scheme of dimension over $S$} \\
	g &= \text{the dimension of $A$ over $S$, assumed to be nonzero.}
\end{align*}
For every $n$, we have the following exact sequence of $\pi$-formal 
group schemes
\begin{equation}
0 \map N^n \stk{i}{\map} J^nA \stk{u}{\map} A \map 0
\end{equation}
We can consider the set $\Hom(N^n,\hG)$ of $\pi$-formal group scheme morphisms
from $N^n$ to $\hG$.
As in section~\ref{subsec-char-groups}, we
endow $\Hom(N^n,\hG)$ with the $R$-module structure given by pointwise addition and scalar
multiplication.

Let $\mathrm{pr}_j: \hG^g \map \hG$ denote the $j$-th projection for $1 \leq j \leq g$. 
In lemma $2.3$ of \cite{bui95}, Buium constructs a $\pi$-formal group 
scheme isomorphism
	$$
	\Psi_1:N^1 \map \hG^g,
	$$
depending only on a choice of \'{e}tale coordinates on $A$. 
Let $\Psi_1$ be the $g$-tuple 
\begin{equation}
\label{Psituple}
\Psi_1 = (\Psi_{11},\dots, \Psi_{1g}),
\end{equation}
where $\Psi_{1j}:= \mathrm{pr}_j \circ \Psi_1$. Then $\{\Psi_{1j}\}$ 
form an $R$-basis of $\Hom(N^1,\hG)$. 

For all $1 \leq i \leq n$ and $1 \leq j \leq g$ define $\Psi_{ij}:N^n
\map \hG$ to be the composition
\begin{equation}
	N^n \stk{\mfrak{f}^{\circ(i-1)}}{\map} N^{n-i+1} \map N^1 \stk{\Psi_{1j}}{\map} \hG
\end{equation}
where the middle unlabeled map is the usual projection $u^{n-i}$ and 
$\Psi_{1j}$ are as in (\ref{Psituple}).

\begin{lemma}
\label{indeplemma}
For all $i$, consider the set $\mbB_i= \{h_{i1},\dots, h_{ig}\}$ where $h_{ij} \in 
k[\bx_1,\dots, \bx_n]$ satisfy
$$h_{ij} \equiv x_{1j}^{q^{i-1}} \bmod (\deg~ (q^{i-1}+1)).$$
Then the set $\mbB_1 \cup \cdots \cup \mbB_n$ are $k$-linearly independent
polynomials. 
\end{lemma}
\begin{proof}
For each $i$, let $\Lambda_i=(\lam_{i1},\dots, \lam_{ig}) \in k^g$ be $n$ 
vectors of length $g$ such that 
$$\Lambda_1. \mbB_1 + \cdots + \Lambda_n . \mbB_n = 0.$$
Then we want to show 
that $\Lambda_i = 0$ for all $i=1, \dots , n$. We will show this by induction
on $i$. Now
$$\Lambda_1. \mbB_1 + \cdots + \Lambda_n.\mbB_n \equiv (\lam_{11}x_{11} + \cdots
+ \lam_{1g} x_{1g}) \bmod (\deg~ 2).$$
Since the left hand side of the above expression is $0$, we must have 
$\lam_{11}= \cdots = \lam_{1g} = 0$ and hence $\Lambda_1 =0$.
Now assume true for $i$. Then we have 
$$\Lambda_{i+1}. \mbB_{i+1} + \cdots + \Lambda_n. \mbB_n \equiv 
\lam_{(i+1)1} x_{11}^{q^i} + \cdots + \lam_{(i+1)g} x_{1g}^{q^i} \bmod 
(\deg~ (q^i+1))$$
where the left hand side of the above expression is $0$ and hence 
we must have $\lam_{(i+1)1} =\cdots = \lam_{(i+1)g} = 0$ and hence 
$\Lambda_{i+1} =0$ and we are done.
\end{proof}

\begin{proposition}
\label{Nrk}
The family $\{\Psi_{ij}\}_{ij}$, for $1\leq i \leq n$ and $ 1\leq j \leq g$, 
forms an 
$R$-basis for $\Hom(N^n,\hG)$.  In particular $\rk_R \Hom(N^n,\hG) = ng$.
\end{proposition}
\begin{proof}
By pulling back via the canonical projection map $u: N^n \map N^{n-1}$ we 
get the following exact sequence of $R$-modules
$$ 0\map \Hom(N^{n-1},\hG) \stk{u^*}{\map} \Hom(N^n,\hG) \map \Hom(
\mb{Ker}(N^n\stk{u}{\map} N^{n-1}),\hG).$$
Note that we have a canonical isomorphism 
$\mb{Ker}(N^n\stk{u}{\map}N^{n-1}) \simeq \mb{Ker}(J^nA \stk{u}{\map}
J^{n-1}A)$ of $\pi$-formal group schemes. As in the proof of corollary 2.10
in \cite{bui95}, we have $\rk_R\Hom(\mb{Ker}(N^n\stk{u}{\map} N^{n-1}),\hG)
\leq g$. Since $N^1 \simeq \hG^g$, we have $\Hom(N^1,\hG) \simeq R^g$. Hence
we have that $\rk_R \Hom(N^n,\hG) \leq ng$.
And by corollary 2.4 and proposition $2.5$ in \cite{bui95} we have that
$\rk_R \Hom(N^n,\hG) \geq ng$. Hence we have $\rk_R \Hom(N^n,\hG) =ng$.

Now we want 
to show that $\{\Psi_{ij}\}_{ij}$ form an $R$-basis of $\Hom(N^n,\hG)$.
By Nakayama's lemma, it is enough to show that the reduction modulo 
$\pi$ of $\Psi_{ij}$, denoted $\oPsi_{ij}$, are $k$-linearly independent.
Then for $\oPsi_{1}= (\oPsi_{11},\dots, \oPsi_{1g})$, 
the multivariate polynomials with $k$ coefficients 
$\oPsi_{1j}$ satisfy
$$\oPsi_{1j}(\bx_1,\dots, \bx_n) \equiv x_{1j} \bmod (\mb{deg } 2).$$
(Here we drop the unnecessary initial coordinate $\bx_0=e$ from the notation.)
The above follows from the general property of the logarithm of a formal 
group law (e.g. see the proof of lemma $2.3$ in \cite{bui95}).
Hence for $\oPsi_i= (\oPsi_{i1},\dots , \oPsi_{ig})$ we have 
$$\oPsi_{ij}(\bx_1,\dots, \bx_n) \equiv x_{1j}^{q^{i-1}} \bmod (\deg~ 
(q^{i-1}+1)).
$$
Hence by lemma \ref{indeplemma}, we have that $\{\oPsi_{ij}\}$ are $k$-linearly 
independent.
\end{proof}
 
We define  $\bPsi_i:= (\Psi_{i1},\cdots , \Psi_{ig}) \in 
\Hom(N^n,\hG)^g$.
Then by proposition \ref{Nrk}, any morphism $\Psi \in \Hom(N^n,\hG)$
can be represented as
\begin{equation}
\Psi = \gamma_1 . \bPsi_1 + \cdots + \gamma_n . \bPsi_n
\end{equation}
where $\gamma_i \in R^g$.

\section{Extensions of abelian schemes and de Rham cohomology}
\label{Fiso}

The aim of the section is to define a certain $K$-module $\bH(A)_K$. We will prove 
that it is naturally an isocrystal admitting a filtration in the consequent sections.
In the process, we will also provide an overview of some background material
on the universal vectorial extension of $A$ and its relation with the 
de Rham cohomology of $A$.

Let $\bx_0=(x_{01},\cdots, x_{0g})$ be an \'{e}tale coordinate system of $A$ 
around the identity section. Then recall from section \ref{localWitt}  
that $\bx= (\bx_0,\dots, \bx_n)$ denotes the local Witt 
coordinates 
for $J^nA$ around the identity section where $\bx_i=(x_{i1},\cdots, x_{ig})$.

In this coordinate system, if
the lift of Frobenius map $\phi: J^{n+1}A \map J^nA$ is given by 
	\begin{equation}
	\phi(\bx_0,\cdots,\bx_{n+1})= (\by_0,\cdots, \by_n)
	\end{equation}
then we have
	$$
	\sum_{i=0}^{n+1} \pi^i x_i^{q^{n+1-i}} = \sum_{i=0}^n \pi^i y_i^{q^{n-i}}
	$$
and hence
	$$
	\sum_{i=0}^{n+1} \pi^i q^{n+1-i}x_i^{q^{n+1-i}-1}dx_i = \sum_{i=0}^n \pi^i q^{n-i}y_i^{q^{n-i}-1}dy_i.	
	$$
In the cotangent space along the identity section $\bx=\by=0$,
we then have $\pi^{n+1} dx_{n+1} = \pi^n dy_n$ and hence
	$$
	dy_n = \pi dx_{n+1}.
	$$
Therefore the derivative matrix of $\phi$ at the identity section is given by a block matrix of size $n\times (n+1)$ with $g\times g$ blocks
\begin{equation}
\label{Dphi}
D\phi= \left(\begin{array}{c c c c c}
\mathbf{0}_g & \pi \mathbbm{I}_g & \mathbf{0}_g & \cdots & \mathbf{0}_g \\
\mathbf{0}_g & \mathbf{0}_g & \pi \mathbbm{I}_g & \cdots & \mathbf{0}_g \\
\vdots  & \vdots & \vdots & \ddots  & \vdots\\
\mathbf{0}_g & \mathbf{0}_g & \mathbf{0}_g & \cdots & \pi \mathbbm{I}_g
\end{array}\right)
\end{equation}
where $\mathbf{0}_g$ denotes the $g\times g$ zero matrix and $\mathbbm{I}_g$ is the $g\times g$ identity.

\subsection{}
Recall that $\Ext^\sharp(A,\hG)$ parametrizes isomorphism classes of 
extensions of $A$ by $\hG$, along with a splitting of the corresponding 
short exact sequence of the (commutative) Lie algebras. (See~\cite{MazMess}, p.\ 13--14, where it
would be denoted $\mathrm{Extrig}(A,\hG)$.) The $R$-module
$\Ext^\sharp(A,\hG)$ is canonically isomorphic
to the de Rham cohomology $\Hdr$ of the abelian scheme
$A$. Here we will give a brief reminder of this identification.

Let $A^\sharp$ denote the universal vectorial extension of $A$. Then 
$A^\sharp$ sits in a short exact sequence of $\pi$-formal 
schemes
\begin{equation}
\label{Asharp}
0 \map V\map A^\sharp \map A \map 0
\end{equation}
where $V$ is the vector group associated to $R$-module 
$H^0(\dualav{A},\Omega_{\dualav{A}})=\dualmod{H^1(A,\Omega_A)}$.
(See \cite{MazMess}, p.\ 24, (2.6.7).)
The universal extension of $A$ satisfies the property that given any
extension $C$ of $A$ by $\hG$ then $C$ can be obtained as a push out of the
universal extension by a unique morphism $f:V \map \hG$:
\begin{equation}
\label{Asharp2}
	\xymatrix{
	0 \ar[r] & V \ar[d]_-f \ar[r] & A^\sharp \ar[d]_-{f_C} \ar[r] \ar[r] 
	& A \ar@{=}[d] \ar[r] & 0 \\
	0 \ar[r] & \hG \ar[r] & C \ar[r] & A \ar[r] & 0  
	}
\end{equation}
where $f_C:A^\sharp \map C$ denotes the evident push-out morphism.
Now consider the short exact sequence of Lie algebras associated to 
(\ref{Asharp}):
\begin{equation}
\label{Asharp3}
0 \map \Lie V \map \Lie  A^\sharp {\map} \Lie A \map 0.
\end{equation}
Then we have identifications $\Lie A \simeq
\dualmod{H^0(A,\Omega_{{A}})}$ and $\Lie A^\sharp \simeq \dualmod{\Hdr}$ (by \cite{MazMess}, p.\ 48, (4.1.7)). 
Under these identifications, (\ref{Asharp3}) becomes the dual of Hodge filtration sequence 
\begin{equation}
\label{Asharp4}
0 \map \dualmod{H^1(A,\Ou_A)} \map \dualmod{\Hdr}  \map \dualmod{H^0(A,\Omega_A)} \map 0.
\end{equation}

Now suppose we are given a cohomology class $\eta\in\Hdr$. Consider the composition
\begin{equation}
\label{map:functional}	
\dualmod{H^1(A,\Ou_A)} \to \dualmod{\Hdr} \to R
\end{equation}
where the first map is the canonical inclusion and the second map is the dual of $\eta$.
The universal property of the universal vectorial extension gives an extension
\begin{equation}
	\xymatrix{
	0 \ar[r] & \hG \ar[r] & C_{\eta} \ar[r] & A \ar[r] & 0. 
	}
\end{equation}
The corresponding extension of tangent spaces is then canonically split,
simply because the functional (\ref{map:functional}) factors through $\dualmod{\Hdr}$. Indeed, 
first consider the universal case, which is to say the extension
induced by the map $\dualmod{H^1(A,\Ou_A)} \to \dualmod{\Hdr}$:
\begin{equation}
	\xymatrix{
	0 \ar[r] & H \ar[r] & (A^\sharp\times H)/V \ar[r] & A \ar[r] & 0. 
	}
\end{equation}
Here, $H$ is the vector group corresponding to the $R$-module $\dualmod{\Hdr}$.
The corresponding extension of tangent spaces is then
\begin{equation}
	\xymatrix{
	0 \ar[r] & \dualmod{\Hdr} \ar[r] & \frac{\Lie(A^\sharp)\oplus\dualmod{\Hdr}}{\dualmod{H^1(A,\Ou_A)}} 
		\ar@/^/[l]^{\sigma} \ar[r]  & \Lie(A) \ar[r] & 0,
	}
\end{equation}
where $\sigma$ is splitting induced by the canonical isomorphism $\Lie(A^\sharp)\isomap\dualmod{\Hdr}$
(of \cite{MazMess}, p.\ 48, (4.1.7).)
Finally, pushing out further along $\dualmod{\Hdr}\to R$ gives our desired splitting $\sigma_\eta$:
\begin{equation}
	\xymatrix{
	0 \ar[r] & R \ar[r] & \Lie(C_\eta) \ar@/^/[l]^{\sigma_\eta} 
	\ar[r]  & \Lie(A) \ar[r] & 0.
	}
\end{equation}
Also observe that if the extension $C_\eta$ itself is split, then the set of splittings of the
tangent space is identified with $\dualmod{\Lie(A)}$.

Summing up, we have a canonical map 
$$
\Hdr \to \Ext^\sharp(A,\hG), \quad \eta\mapsto (C_\eta,\sigma_\eta).
$$
It is an isomorphism because it sits in a morphism of short exact sequences
\begin{equation}
	\xymatrix{
	0 \ar[r] & H^0(A,\Omega_A) \ar[d] \ar[r] & \Hdr \ar[d] \ar[r] & H^1(A,\Ou_A) \ar[d] \ar[r] & 0 \\
	0 \ar[r] & \dualmod{\Lie(A)} \ar[r] & \Ext^\sharp(A,\hG) \ar[r] & \Ext(A,\hG) \ar[r] & 0,  
	}
\end{equation}
which is an isomorphism on each end.

\subsection{}
For all $n \geq 1$, we will define maps from $\HomA(N^n,\hG)$ to 
$\Ext^\sharp(A,\hG)$. These maps are obtained by push-outs of $J^nA$ by $\Psi \in \Hom(N^n,\hG)$. 
Consider the exact sequence 
	\begin{equation}
		\label{jet-seq}
	0 \map N^n \stk{i}{\map} J^nA \stk{u}{\map} A \map 0	
	\end{equation}
Given a character $\Psi \in\HomA(N^n,\hG)$ consider the push out
	$$
	\xymatrix{
	0 \ar[r] & N^n \ar[d]_-\Psi \ar[r]^i & J^nA \ar[d]_-{g_\Psi} \ar[r]^u \ar[r] 
	& A \ar@{=}[d] \ar[r] & 0 \\
	0 \ar[r] & \hG \ar[r]^i & A^*_\Psi \ar[r] & A \ar[r] & 0 \\  
	} 
	$$
where $A^*_\Psi = \frac{J^nA \times \hG}{\Gamma(N^n)}$ and
$\Gamma(N^n) = \{(i(z),-\Psi(z)) | ~ z \in N^n\} \subset J^nA \times N^n$
and $g_\Psi(x)= [x,0] \in A^*_\Psi$.

Based on the choice of local \'{e}tale coordinates $\bx_0$ for $A$, consider
the local Witt coordinates for $J^nA$ defined in section 5. 
This gives us a basis for $\Lie J^nA$ which
we will still denote as $\bx=(\bx_0, \cdots, \bx_n)$. Let $\switt: \Lie J^nA 
\map \Lie N^n$ be given by $\switt (\bx_0, \cdots , \bx_n)= (\bx_1,\cdots,
\bx_n)$.
Thus we have the following split exact sequence of $R$-modules
	$$
	\xymatrix{
	0 \ar[r] & \Lie N^n \ar[r]^-{Di} & \Lie J^nA \ar@/^/[l]^{{\switt}}
	\ar[r]^-{Du} & \Lie(A) \ar[r] & 0
	}
	$$
Let $v: \Lie A \map \Lie J^nA$ be the corresponding section, satisfying
$\switt= \mathbbm{1} - v \circ Du$. Then in Witt coordinates, $v$ is given by $v(\bx_0) = (\bx_0,0,
\cdots, 0)$.

Let $s_\Psi$ denote the induced splitting of the push out extension
	$$
	\xymatrix{
	0 \ar[r] & \Lie \hG \ar[r] & \Lie(A^*_\Psi) \ar@/^/[l]^{s_\Psi}
	\ar[r] & \Lie(A) \ar[r] & 0
	}
	$$
It can be described explicitly in terms of the composition
	$$
	\tilde{s}_\Psi : \Lie J^nA \times \Lie\hG \longmap \frac{\Lie  J^nA \times \Lie  \hG}{\Lie \Gamma(N^n)} 
	\longlabelmap{s_\Psi} \Lie \hG
	$$
by
	$$
	\tilde{s}_\Psi(\bx,y)= D\Psi({\switt}(\bx)) + y.
	$$
We then have a morphism of exact sequences
	\begin{equation}
	\label{trivext}
	\xymatrix{
		0 \ar[r] & \bX_n(A) \ar[r] \ar[d] & 
		\HomA(N^n,\hG) \ar[r] \ar[d]_{\Psi\mapsto (A_\Psi^*,s_\Psi)}& 
		\Ext(A,\hG)  \ar@{=}[d]& \\
		0 \ar[r] &\dualmod{\Lie(A)} \ar[r] & \Ext^\sharp(A,\hG) \ar[r] & 
\Ext(A,\hG) \ar[r] & 0
	}
	\end{equation}

The following proposition describes the morphism $\bX_n(A)\map \dualmod{\Lie(A)}$ explicitly.

\begin{proposition}\label{pro:formulas}
Let $\Theta$ be a character in $\bX_n(A)$, and put $\Psi=i^*\Theta\in
\Hom(N^n,\hG)$. 
\begin{enumerate}
	\item The map $\bX_n(A)\to \dualmod{\Lie(A)}$ of (\ref{trivext}) sends $\Theta$	to $-D\Theta \circ v $.
	\item Let $\tilde{\Theta}=\phi^*\Theta$, then $-D\tilde{\Theta} \circ v = 0$.
	\item If $\Psi \in i^*\phi^*(\bX_n(A))$, then the class $(A^*_\Psi,s_\Psi)\in \Ext^\sharp(A,\hG)$ is zero.
\end{enumerate}
\end{proposition}

\begin{proof}
(1): 
Let us recall in explicit terms how the map is given. For the split extension 
$A\times\hG$,
the retractions $\Lie(A) \times\Lie \hG = \Lie(A\times\hG) \to \Lie \hG$ are in bijection with maps 
$\Lie(A)\to\Lie\hG$, a retraction $s$ corresponding to map $\bx_0 \mapsto s(\bx_0,0)$. 
Therefore to determine
the image of $D\Theta$, we need to identify $\Lie A^*_\Psi$ with the
split extension and then apply this map to $s_\Psi$.

The trivialization $A^*_\Psi \isomap A \times \hG$ is given by the map
	$$
	\frac{J^nA \times \hG}{\Gamma(N^n)} = A^*_\Psi 
	\longisomap A\times \hG
	$$
defined by $[\ba,b]\mapsto (u(\ba),\Theta(\ba)+b)$.
So the trivialization of the extension of Lie algebras is given by the map
	$$
	\frac{\Lie J^nA \times \Lie \hG}{\Lie \Gamma(N^n)} = \Lie A^*_\Psi 
\longisomap \Lie(A)\times \Lie \hG
	$$
defined by $[\ba,b]\mapsto (u(\ba),D\Theta(\ba)+b)$. The inverse isomorphism 
$H$ is then given by the expression
	$$
	H(\bx,y) = [v(\bx),y-D\Theta(v(\bx))],
	$$
and so the composition $\Lie (A)\to \Lie(A) \times \Lie \hG \to \Lie(A^*_\Psi)
 \to \Lie \hG$ is simply $-D\Theta\circ v$.

(2): We have $\tilde{\Theta} = \Theta \circ \phi$ and hence $D\tilde{\Theta}
\circ v = D\Theta \circ D\phi \circ v$. But note that by equation (\ref{Dphi}),
 $D\phi \circ v =0$ and hence we are done.

(3):
We know from diagram (\ref{trivext}) that $A^*_\Psi$ is a trivial extension
since ${\Psi}$ lies in $i^* \bX_{n+1}(A)$. Now as in part (2) of proposition~\ref{pro:formulas},
we have, in the notation of that proposition, $-D\tilde{\Theta} \circ v = 0$ 
and therefore the class in $\Ext^\sharp(E,\hG)$ is zero by part (1).

\end{proof}

\begin{proposition}
\label{pidiv}
If $\Theta \in \bX_{n-1}(A)$, then $(\phi \circ i - i \circ \mfrak{f})^*\Theta=
\pi \Psi$ for some $\Psi \in \Hom(N^1,\hG)$.
\end{proposition}
\begin{proof}
By theorem \ref{comfact}, we have
		$$
		(\phi \circ i - i \circ \mfrak{f})^*\Theta \in \Hom(N^1,\hG).
		$$
It is then enough to show that this character
vanishes modulo $\pi$, which is to say it restricts to $0$ on the fiber modulo $\pi$.
But this holds because both $\phi$ and $\mfrak{f}$ reduce to the Frobenius map modulo $\pi$.
\end{proof}

\subsection{Filtered Isocrystal.}

Now we review some of the basic definitions of filtered isocrystals in 
$p$-adic Hodge theory.
For more details, we refer to the excellent survey article \cite{BriCon}.

A {\em filtered module} over a commutative ring $R$ is an $R$-module $M$ 
endowed with a collection $\{M^i\}_{i\in \bZ}$ of submodules that is 
decreasing in the sense that 
$M^{i+1} \subset M^i$ 
for all $i \in \bZ$. 
The category of such objects will be denoted as $\mb{Fil}_R$.

Let us define a {\em (non-degenerate) isocrystal} over $K$ to be a finite dimensional 
$K$-vector space $D$
equipped with a (bijective) Frobenius-semilinear endomorphism $\phi_D:D \map D$.
The category of isocrystals over $K$ is denoted by $\mb{Mod}^{\phi}_K$ and
forms an abelian category.

A {\em filtered isocrystal} over $K$ is a triple $(D,\phi_D,\{D^i\})$
where $(D,\phi_D)$ is an isocrystal over $K$ and $(D,\{D^i\})$ is an
object in $\mb{Fil}_K$. 
It is called
{\em weakly admissible} if for every subobject $(D',\phi',\{{D'}^i\})$
we have
\begin{equation}
\label{weakadm}
\sum i . \dim_K (D'^i/D'^{i+1}) \leq {\mb{ord}}_p\det(\phi')
\end{equation}
and if for $(D',\phi',\{{D'}^i\}) = (D,\phi,\{D^i\})$ we have equality in
this relation.

Now we will define a $K$-module $\bH(A)_K$ which will have a semilinear 
endomorphism. We will eventually prove (corollary \ref{filtisocor})
that this is a filtered isocrystal.

For any $R$-module $M$, let us fix the notation
	$$
	M_{\phi} = R\otimes_{\phi,R}M.
	$$
Then the $\phi$-linear map $\phi^*:\bX_{n-1}(A) \to \bX_n(A)$ induces a linear map
$\bX_{n-1}(A)_{\phi}\to \bX_n(A)$, which we will abusively also denote $\phi^*$.
We then define
	$$
	\bH_n(A) = \frac{\Hom(N^n,\hG)}{i^*\phi^*(\bX_{n-1}(A)_{\phi})}
	$$
Then 
$u:N^{n+1}\map N^n$ induces $u^*:\Hom(N^n,\hG) \map \Hom(N^{n+1},\hG)$. And since $u$ commutes with
both $i$ and $\phi$, we have
	$$
	u^*i^*\phi^*(\bX_n(A)) = i^*\phi^*u^*(\bX_n(A)) \subset i^*\phi^*(\bX_{n+1}(A)),
	$$
and hence $u$ also induces a map $u^*:\bH_n(A) \map \bH_{n+1}(A)$.
Define 
\begin{equation}
\label{bigH}
\bH(A)= \varinjlim \bH_n(A)
\end{equation}
 where the limit is taken in the category
of $R$-modules.

Similarly, $\mfrak{f}: N^{n+1} \map N^n$ induces $\mfrak{f}^*:\Hom(N^n,\hG) \map \Hom(N^{n+1},\hG)$, 
which descends to a $\phi$-linear morphism of $R$-modules
\begin{equation}	
\label{latfrobH}
\mfrak{f}^*:\bH_n(A) \map\bH_{n+1}(A)
\end{equation}
because by theorem \ref{latfrob} we have $\mfrak{f}^*i^*\phi^*(\bX_{n-1}(A))= 
i^*\phi^*\phi^*(\bX_{n-1}(A)) \subset i^*\phi^*\bX_{n}(A)$.
This in turn induces a $\phi$-linear endomorphism $\mfrak{f}^*:\bH(A) \map \bH(A)$. 

\begin{proposition}
\label{diff}
For any character $\Theta$ in $\bX_n(A)$, let the derivative at the identity
with respect to our chosen coordinates 
be $D\Theta = (A_0, \cdots, A_n)$ where $A_j \in \mb{Mat}_{1\times g} (R)$.
\begin{enumerate}
    \item We have
            $$
            i^*\phi^*\Theta = \mfrak{f}^*(i^*\Theta)+ \gamma. \bPsi_1,
            $$
            where $\gamma=\pi A_0$. 
    \item For $n\geq 2$, we have
            $$
            i^*(\phi^{\circ n})^*\Theta= (\mfrak{f}^{n-1})^* i^*\phi^*\Theta.
            $$
\end{enumerate}
\end{proposition}

\begin{proof}

(1): 
By theorem~\ref{comfact}, there is an element $\gamma \in R^g$ such that 
	$$
	(i^* \phi^* - \mfrak{f}^* i^*)\Theta = \gamma . \bPsi_1
	$$
Then we have
	$$
	i^*\phi^* \Theta \equiv \gamma . \bPsi_1  
	\bmod (\Psi_{21}, \dots ,\Psi_{2g},\dots, \Psi_{(n+1)1},\dots, \Psi_{(n+1)g})
	$$
in $\Hom(N^{n+1},\hG)$. 
By equation (\ref{Dphi}) the derivative matrix 
$D(\Theta \circ \phi \circ i)$ at the identity section is 
\begin{equation}
(\pi A_{0}, \cdots, \pi A_n)
\end{equation}
Hence we have 
\begin{equation}
\label{thegam}
\gamma = \pi A_{0}.
\end{equation}

(2): This  follows directly from theorem \ref{latfrob}. 
\end{proof}

\begin{proposition}
\label{pro-filt}
For any $n\geq 0$, the diagram
    $$
    \xymatrix{
    \bX_n(E)/\bX_{n-1}(E) \ar@{^{(}->}^{\phi^*}[r]\ar@{^{(}->}^{i^*}[d] 
            & \bX_{n+1}(E)/\bX_n(E) \ar@{^{(}->}^{i^*}[d] \\
    \HomA(N^n,\hG)/\HomA(N^{n-1},\hG) \ar^{\mfrak{f}^*}_{\sim}[r] 
            & \HomA(N^{n+1},\hG)/\HomA(N^n,\hG)
    }
    $$
is commutative. The morphisms $i^*$ and $\phi^*$ are injective, and 
$\mfrak{f}^*$ is bijective.
\end{proposition}

\begin{proof}
For $n\geq 1$, commutativity of the diagram follows from 
proposition~\ref{diff}; for $n=0$, it is immediate since $\bX_0(A) = 0$.

The maps $i^*$ are injective because the projections $J^nA \to J^{n-1}A$ and 
$N^n\to N^{n-1}$ have the same
kernel, and $\mfrak{f}^*$ is an isomorphism by proposition~\ref{Nrk}. 
It follows that  $\phi^*$ is an injection.
\end{proof}

\section{Finite generation of $\bX_\infty(A)_K$ as a $K\{\phi^*\}$-module}
\label{exactsequence}
For every $n$ we have the following short exact sequence
\begin{equation}
\label{se}
0 \map N^n \map J^nA \map A \map 0
\end{equation}
Applying $\Hom(-, \hG)$ to the above short exact sequence gives us
\begin{equation}
\label{sel}
0 \map \bX_n (A) \map \Hom(N^n,\hG) \stk{\partial}{\map} \Ext(A,\hG) 
\end{equation}
By the theory of extensions of groups that admit rational sections (see 
\cite{Serre}, page 185, theorem 7) we have 
$\Ext(A,\hG)\simeq H^1(A,\Ou_A) \simeq R^g$. 

Let $\bI_n(A) := \mb{image}(\partial)$. 
Note that since for all $n$, there are maps $\Hom(N^n,\hG) \stk{u^*}{\inj} \Hom(N^{n+1},\hG)$,
we have $\bI_n(A) \subset \bI_{n+1}(A)$. Define 
\begin{equation}
\label{bigI}
\bI(A):= \varinjlim \bI_n(A)
\end{equation}
and 
$$
h_i= \rk \bI_i(A) - \rk \bI_{i-1}(A)
$$
for all $i \geq 1$.
We define the {\it upper splitting number} to be the smallest number 
$\mup\geq 1$ such that $h_n=0$ for all $n \geq \mup$.  
Note that $\mup$ exists since $$\bI_0(A) \subset \bI_{1}(A) \subset \cdots \subset \Ext(A,\hG)\cong R^g.$$

We define the {\it lower splitting number} to be the unique $\mlow$ satisfying
$\bX_{\mlow}(A) \ne \{ 0\}$ and $\bX_{\mlow-1}(A) = \{0\}$.

\begin{lemma}
We have $\mlow =1$ or $\mlow=2$.
\end{lemma}
{\it Proof.} 
We know $\rk_R \Ext(A,\hG) = g$ and
$\rk_R \Hom(N^n,\hG) = ng$, by corollary \ref{Nrk}.
So for $n\geq 2$, the kernel $\bX_n(A)$ in (\ref{sel}) has rank at least $g$, which is at least $1$.
$\qed$

Let us recall the the structure of $\bX_\infty(A)_K$ as a $K\{\phi^*\}$-module where $K\{\phi^*\}$ is the twisted
polynomial ring with commutation law $\phi^* r = \phi(r) \phi^*.$ The $R$-linear structure comes from the fact
that the target of characters is the $\pi$-formal group scheme $\hG$, which is an $R$-module scheme. The
$K\{\phi^*\}$-module structure is then given by extension of scalars from $R$ to $K$. The operator $\phi^*$ acts
on a character $\Theta \in \bX_n(A)$ via pull-back $\Theta \mapsto \phi^* \Theta$ as defined in
(\ref{phi-pull-back}).

The operators $u^*$ and $\phi^*$ induce a bi-filtration on $\bX_n(A)$:
$$
\xymatrix{
\bX_1(A) \ar^{u^*}[r] & \bX_2(A) \ar^{u^*}[r] & \cdots\ar^{u^*}[r] & \bX_{n-1}(A)\ar^{u^*}[r] & \bX_n(A) \\
0\ar[r]\ar[u]		   & \bX_1(A)_{\phi} \ar^{u^*}[r]\ar^{\phi^*}[u] & \cdots\ar^{u^*}[r] 
	& \bX_{n-2}(A)_{\phi}\ar^{u^*}[r]\ar^{\phi^*}[u] & \bX_{n-1}(A)_{\phi}\ar^{\phi^*}[u] \\
		   & 			& \ddots & \vdots \ar^{\phi^*}[u] & \vdots\ar^{\phi^*}[u] \\
		   &			&		 &	\bX_1(A)_{\phi^{n-2}}\ar^{\phi^*}[u]\ar^{u^*}[r] 
	& \bX_2(A)_{\phi^{n-2}}\ar^{\phi^*}[u]		   \\
		   &			&		 &	0\ar[r]\ar[u]			 & \bX_1(A)_{\phi^{n-1}}\ar^{\phi^*}[u]		   
}
$$
All the morphisms are $R$-linear and injective, 
and all the squares are pull-back squares, by proposition~\ref{pro-filt}.
We say a differential character $\Theta \in \bX_n(A)_K$ is {\it primitive} if 
	$$
	\Theta \notin u^* \bX_{n-1}(A)_K + \phi^* (\bX_{n-1}(A)_K)_{\phi}.
	$$
Below, we will write $\mathbbm{B}_i$ for a subset of $\bX_i(A)_K$ which
is a lift of a $K$-basis of
$$
\bX_i(A)_K/(u^*\bX_{i-1}(A)_K + \phi^* (\bX_{i-1}(A)_K)_{\phi}).
$$  
We will call $\mathbbm{B}_i$  a {\it primitive basis} for $\bX_i(A)_K$.
If $\mathbb{B}_1, \dots \mathbb{B}_n$ are primitive bases for $\bX_1(A)_K,\dots,\bX_n(A)_K$, then
it follows from the diagram above that $\bX_n(A)_K$ has a $K$-basis
	$$
	S_n(\mathbbm{B}_{1})\cup \cdots \cup S_n(\mathbbm{B}_{n})
	$$
where
$$
S_n(\mathbbm{B}_i)= \{\phi^{*h}\Theta~|~\mb{for all } 0 \leq h \leq (n-i) \mb{ and } \Theta \in \mathbbm{B}_i\} 
$$
In particular, if we denote the dimensions of the associated bi-graded pieces by
	\begin{equation}
		\label{neo2}
	l_i = \rk \bX_i(A)_K /(u^*\bX_{i-1}(A)_K  + \phi^* (\bX_{i-1}(A)_K)_{\phi}),		
	\end{equation}
then we have
	\begin{equation}
		\label{neo}
	\rk \bX_n(A)_K = l_n + 2l_{n-1} + \cdots + nl_1.
	\end{equation}

\begin{lemma}
\label{diff1}
For all $n \geq 2$, $l_n = h_{n-1} - h_n$.
\end{lemma}
\begin{proof}
From the exact sequence
\begin{equation}
0 \map \bX_n (A)_K \map \Hom(N^n,\hG)_K \stk{\partial}{\map} \bI_n(A)_K \map 0,
\end{equation}
we have $ \rk_K(\bI_n(A)_K)=ng-\rk(\bX_n (A)_K)$.
Combining this with (\ref{neo}), we have
\begin{align*}
	h_n &= (ng-\rk(\bX_n (A)_K)) - ((n-1)g - \rk(\bX_{n-1} (A)_K)) \\
		&= g - (l_1+\cdots l_n)
\end{align*}
and hence $h_{n-1}-h_n=l_n$.
\end{proof}

\begin{lemma}
\label{decrease}
For all $n\geq 1$, $h_n$ is a (weakly) decreasing function of $n$.
\end{lemma}
\begin{proof}
	This follows from lemma \ref{diff1} since $l_n\geq 0$ for all $n$.
\end{proof}

\begin{corollary}
\label{hua}
If $h_N=0$, then $l_n =0$ for all $n \geq N+1$.
\end{corollary}
\begin{proof}
	This follows from lemmas \ref{diff1} and \ref{decrease}.
\end{proof}

\begin{proposition}
\label{lowandup}
(1) The largest integer $n$ such that $l_n \neq 0$ is $n=\mup$. In other words, the
largest order of a primitive character is $\mup$.

(2) The smallest order of a primitive character is $\mlow$. 
In particular, $\mlow \leq \mup$.
\end{proposition}
\begin{proof} (1):
By definition of
$\mup$, we have $h_{\mup} = h_{\mup+1} = \cdots = 0$. This implies by 
lemma \ref{diff1} that $l_n =0$ for all $n\geq \mup +1$. On the other hand,
$h_{\mup-1}\geq 1$ and hence $l_{\mup}=h_{\mup-1}-h_{\mup} \geq 1$.
Thus there is a primitive character of order $\mup$.

(2): 
Since $\bX_{\mlow-1}(A)_K=\{0\}$ and $\bX_{\mlow}(A)_K\neq \{0\}$, there exists a primitive
character of order $\mlow$ and no smaller order.
\end{proof}

Now consider the integers $i$ for which there exists a primitive character of order $i$. By the previous result,
they fall into a sequence
	$$
	\mlow = i_1 < \cdots < i_r = \mup.
	$$
This sequence forms an interesting numerical invariant of $A$. A more refined invariant would include
the multiplicities $l_{i_1},l_{i_2},\dots,l_{i_r}$.

\begin{theorem}
\label{maxorder}
For any abelian scheme $A$ of dimension $g$, $\bX_\infty(A)_K$ is freely
$K\{\phi^*\}$-generated by $g$ differential characters of order at most $g+1$. 
Thus
$l_{\mlow}+\cdots+ l_{\mup}=g$ and $\mup \leq g+1$.
\end{theorem}
\begin{proof}
Note that $\bX_\infty(A)_K$ is freely generated by the 
primitive characters $\mathbb{B}_{1}\cup\cdots\cup\mathbb{B}_{\mup}$.
The number of them is
	\begin{align*}
	l_1+\cdots+l_{\mup} &= l_1 + (h_1-h_2)+ \cdots + (h_{\mup-1}-h_{\mup}) \\
		&= l_1 + h_1 \\
		&= g.
	\end{align*}

Now we show their orders are at most $g+1$.
Since $\sum h_i \leq g$ and  $h_i$
are a weakly decreasing sequence of non-negative integers, we must have
$h_i =0$ for all $i \geq g+1$, and hence $\mup \leq g+1$.
\end{proof}

Now define $$\bXp(A):= \varinjlim \bX_n(A)/\phi^*\bX_{n-1}(A)_{\phi}.$$

\begin{corollary}
\label{xprim}
We have 
$$\bXp(A)_K\simeq\bX_{\mup}(A)_K/\phi^* (\bX_{\mup-1}(A)_K)_{\phi} $$
Moreover, $\mathbbm{B}_{i_1} \cup \cdots \cup
\mathbbm{B}_{i_r}$ is a $K$-basis for $\bXp(A)_K$.
\end{corollary}
\begin{proof}
For all $n \geq \mup$,  the subset
$S_n(\mathbbm{B}_{i_1}) \cup \cdots \cup S_n(\mathbbm{B}_{i_r})$ 
generates $\bX_n(A)_K$ as a $K$-module, whereas 
$(S_n(\mathbbm{B}_{i_1})\backslash \mbB_{i_1}) \cup \cdots 
\cup (S_n(\mathbbm{B}_{i_r})\backslash \mbB_{i_r})$ 
generates $\phi^*(\bX_{n-1}(A)_K)_{\phi}$ as a $K$-module. Therefore 
$\bX_n(A)_K/ \phi^*(\bX_{n-1}(A)_K)_{\phi}$ is generated by $\mbB_{i_1} \cup \cdots \cup
\mbB_{i_r}$ as a $K$-module for all $n \geq \mup$ and hence in particular
	$$
	\bXp(A)_K \simeq \bX_{\mup}(A)_K/\phi^*(\bX_{\mup -1}(A)_K)_{\phi}.
	$$
\end{proof}

\begin{corollary}
\label{mell}
If $g=1$, then $\mlow= \mup =: m$ and $\bXp(A)_K \simeq \bX_m(A)_K$.
\end{corollary}
\begin{proof}
If $\mlow =1$, then note that $h_1=0$. Since $h_n$ is a weakly decreasing
function in $n$, we have $h_1= h_2 = \cdots = 0$. Therefore the rank of 
$\bI_n(A)$ is $0$ for all $n \geq 0$ and hence $\mup =1$.

If $\mlow=2$, then note that $h_1=1$ since $\partial:\Hom(N^1,\hG)_K \map 
\Ext(A,\hG)_K$ is injective. But since $\rk_K\Ext(A,\hG)_K=1$, we have
$h_2=h_3= \cdots = 0$. Therefore $\rk \bI_i(A)$ is constant for all $i\geq 1$
and hence $\mup=2$.
\end{proof}

\section{The $F$-isocrystal and Hodge sequence of $A$}
In this section, we show that $\bH(A)_K$ is a filtered isocrystal.
Given our choice of \'{e}tale coordinates, we construct 
a canonical $K$-basis of our filtered isocrystal $\bH(A)$ and also show the
exact sequence corresponding to this filtration admits a canonical map to 
the Hodge sequence of $A$ as introduced in (\ref{Asharp3}) or (\ref{Asharp4}).
This will be shown in theorem \ref{isocrys} and corollary \ref{filtisocor}.

\begin{proposition}
\label{Hniso}
The morphism $$u^*: \bH_n(A)_K \map \bH_{n+1}(A)_K$$ is injective. For 
$n \geq \mup$, it is an isomorphism.
\end{proposition}
\begin{proof}
Consider the following diagram of exact sequences:
        $$
        \xymatrix{
        & 0 & 0 \\
        & 
                \frac{(\bX_n(A)_K)_{\phi}}{u^*(\bX_{n-1}(A)_K)_{\phi}}  
\ar[u]\ar[r]^-{i^*\phi^*} 
& \frac{\Hom(N^{n+1},\hG)_K}{\Hom(N^n,\hG)_K \ar[u]} \\
        0 \ar[r] & 
                (\bX_{n}(A)_K)_{\phi} \ar[u] \ar^-{i^*\phi^*}[r] & 
                \Hom(N^{n+1},\hG)_K \ar[u] \ar[r]& \bH_{n+1}(A)_K \ar[r] & 
                0 \\
        0 \ar[r] & 
                (\bX_{n-1}(A)_K)_{\phi} \ar^{u^*}[u] \ar^-{i^*\phi^*}[r] & 
                \Hom(N^{n},\hG)_K \ar^{u^*}[u] \ar[r] & \bH_n(A)_K \ar[r]\ar^{u^*}[u] & 
                0 \\
        & 
                0 \ar[u] & 
                0 \ar[u] &  & 
        }
        $$
Then $i^*\phi^*: \bX_n(A)/\bX_{n-1}(A) \map \Hom(N^{n+1},\hG)/\Hom(N^n,
\hG)$ is injective by proposition \ref{pro-filt} and hence, by snake lemma,
 $\bH_n(A)_K \map \bH_{n+1}(A)_K$ is injective for all $n$.

It remains to show that $u^*:\bH_n(A)_K \map \bH_{n+1}(A)_K$ is 
surjective for all $n \geq \mup$.
Let $\bX_\infty(A)$ be primitively generated by the $g$ primitive characters
$\mathbbm{B}_{\mlow}\cup \cdots \cup \mathbbm{B}_{\mup}$ as in theorem 
\ref{maxorder}. Then the image of $(S_n(\mathbbm{B}_{\mlow})\backslash 
S_{n-1}(\mbB_{\mlow})) \cup \cdots \cup (S_n(\mbB_{\mup}) \backslash 
S_{n-1}(\mathbbm{B}_{\mup}))$
forms a $K$-basis for $\bX_n(A)_K/u^*\bX_{n-1}(A)_K$
for all $n \geq \mup$. Hence $\rk \bX_n(A)_K/u^*\bX_{n-1}(A)_K = g =
\Hom(N^n,\hG)_K/\Hom(N^{n-1},\hG)_K$ which implies $i^*\phi^*$ is surjective
for all $n\geq \mup$ and hence $u^*:\bH_n(A) \map \bH_{n+1}(A)$ is surjective.

\end{proof}

Let us recall the diagram (\ref{trivext})
	\begin{equation}
\nonumber
	\xymatrix{
		0 \ar[r] & \bX_n(A) \ar[r] \ar[d] & 
		\HomA(N^n,\hG) \ar^{\partial}[r] \ar[d]_{\Psi\mapsto (A_\Psi^*,s_\Psi)}& 
		\Ext(A,\hG)  \ar@{=}[d]& \\
		0 \ar[r] &\dualmod{\Lie(A)} \ar[r] & \Ext^\sharp(A,\hG) 
\ar[r] & \Ext(A,\hG) \ar[r] & 0
	}
	\end{equation}
Then by proposition \ref{pro:formulas}(3), $i^*\phi^*(\bX_{n-1}(A)_{\phi})$ is in 
the kernel
of the middle vertical map. Therefore 
this diagram induces a diagram
	\begin{equation}
	\label{diag-crys}
	\xymatrix{
	0 \ar[r] & \frac{\bX_n(A)}{\phi^*(\bX_{n-1}(A)_{\phi})} \ar[r] 
\ar[d]_-\Upsilon & 
	\bH_n(A) \ar[r] \ar[d]^-{\Phi} & 
	\bI_n(A) \ar[r] \ar@{^{(}->}[d]& 0\\
	0 \ar[r] &\dualmod{\Lie(A)} \ar[r] & \Ext^\sharp(A,\hG) \ar[r] & \Ext(A,\hG)
	 \ar[r] & 0
	}
	\end{equation}
where $\bI_n(A)$ denotes the image of $\partial:\Hom(N^n,\hG) \map \Ext_A(A,\hG)$.
Passing to the limit gives a diagram:
\begin{equation}
	\label{diag-crys-limit}
	\xymatrix{
	0 \ar[r] & \bXp(A)_K \ar[d]_\Upsilon \ar[r] & 
	\bH(A)_K \ar[d]_\Phi \ar[r] &\bI(A)_K \ar@{^{(}->}[d] \ar[r] &  0 \\
	0 \ar[r] & \dualmod{\Lie (A)}_K \ar[r] & \Ext^\sharp(A,\hG)_K \ar[r] & \Ext(A,\hG)_K
	 \ar[r] & 0 
	}
\end{equation}

To describe $\Upsilon$ explicitly, let us fix some notation.
Let $\Theta_1, \cdots, \Theta_g$ be a basis for $\bXp(A)_K$. 
For each $\Theta_i: J^nA \map \hG$, let the derivative matrix 
at the identity be $D\Theta_i= (A_{0i}, \dots, A_{ni})$ where $A_{ji}$ are 
$(1\times g)$-matrices. 

Then the following theorem is, at last, the precise form of theorem~\ref{isocrys-intro} from the introduction:
\begin{theorem}
\label{isocrys}
The diagram (\ref{diag-crys-limit}) is a map of short exact sequences of $K$-modules.
We have $\rk_K \bH(A)_K \leq 2g$.
Moreover $\Phi$ is injective if and only if the $g\times g$ matrix $(A_{01},\dots,A_{0g})$
is invertible over $K$. 
\end{theorem}
\begin{proof}
We know the diagram is a map of short exact sequences by the discussion above.
By corollary \ref{xprim}, we have $\rk_K \bXp(A)_K= g$ and
$\rk_K \bI(A)_K \leq g$ and hence $\rk_K \bH(A) \leq 2g$.

Now by proposition \ref{pro:formulas}(1), we have
 $\Upsilon(\Theta_i) = - D\Theta_i \circ v = - A_{0i}.$
Therefore the $g\times g$ matrix of $\Upsilon$ with respect to our basis is 
given by $(A_{01}, \cdots , A_{0g})$, and we are done.
\end{proof}

\begin{corollary}
\label{filtisocor}
The $K$-module $\bH(A)_K$ is a filtered isocrystal of rank at most $2g$ with semilinear endomorphism
$\mfrak{f}^*$ and filtration $\bH(A)_K \supseteq \bXp(A)_K \supseteq \{0\}$.
\end{corollary}

\section{The elliptic curve case}
In this section, we will closely look at the structure of the filtered 
isocrystal $\bH(A)_K$ when $A$ is an elliptic curve. 
We show that $\bH(A)$, defined in (\ref{bigH}), 
is a finitely generated free sub-$R$-module of $\bH(A) \otimes_R K = \bH(A)_K$.

When $A$ is an elliptic curve over $S$, then by corollary \ref{mell} we have
$m= \mlow = \mup \leq 2$. The following are two possible choices of $\Theta_m
\in \bX_m(A)_K$:

If $m=1$, by proposition \ref{pro-filt}, there
exists $\Theta_1 \in \bX_1(A)_K$ such that $i^*\Theta_1= \Psi_1$.

If $m=2$, again by proposition \ref{pro-filt}, there exists $\Theta_2 \in 
\bX_2(A)_K$ such that $i^*\Theta_2 = \Psi_2 -\lam \Psi_1$. 
Since $i^*\Theta_2 \in \ker\left[\partial:\Hom(N^2,\hG) \map \Ext(A,\hG)
\right]$, we have 
$$\partial \Psi_2 = \lam_1 \partial\Psi_1.$$

Note that 
$\lam$ is in $K$ to start with. However, we will show in theorem \ref{intlam} 
that $\lam$ is integral, in other words, $\lam \in R$.
Now pulling back $\Theta_2$ by $\phi$ and $i$ we have
	\begin{equation}
	i^*\phi^* \Theta_2 = \Psi_3 - \phi(\lam) \Psi_2 + \gamma \Psi_1
	\end{equation}
Now again since 
$i^* (\phi^* \Theta_2) \in \ker \left[\partial:\Hom(N^3,\hG)\map \Ext(A,\hG)\right]$ 
we also have
	$$
	\partial \Psi_3 = \phi(\lam) \partial\Psi_2 -\gamma \partial\Psi_1.
	$$

\begin{proposition}
\label{hn}
For $n \geq m$, 
	$$
	\bH_n(A)_K \simeq 
	\left\{\begin{array}{ll}
		K\langle\Psi_1 \rangle, & \mb{if } m=1 \\
		K\langle \Psi_1,\Psi_2\rangle, & \mb{if } m=2
	\end{array} 
	\right.
	$$
\end{proposition}
\begin{proof}
By definition, we know that $\bH_m(A) = \Hom(N^m,\hG)$. Then the result 
follows from the 
above discussion and proposition \ref{Hniso}.
\end{proof}

\begin{proposition}\label{pro:I-basis}
We have
	$$
	\bI_n(A) \otimes K \simeq \left\{\begin{array}{l} 
	K\langle\Psi_1,\dots, \Psi_{n}\rangle,~ \mb{if } n\leq m-1 \\
	K\langle\Psi_1,\dots, \Psi_{m-1}\rangle,~ \mb{if } n\geq m-1
	\end{array}\right.
	$$	
\end{proposition}

\begin{proof}
The case $n\leq m-1$ is clear. So suppose $n\geq m-1$. Then
$\HomA(N^j,\hG)\otimes K$ has basis $\Psi_1,\dots,\Psi_j$,
and $\bX_n(A)\otimes K$ has basis $\Theta_m,\dots,(\phi^{n-m})^*\Theta_m$.
Since each $(\phi^j)^*\Theta_m$ equals $\Psi_{m+j}$ plus lower order terms,
$K\langle\Psi_1,\dots,\Psi_{m-1}\rangle$ is a complement to the subspace $X_n(A)$ of $\HomA(N^n,\hG)$.
Therefore the map $\partial$ from $K\langle\Psi_1,\dots,\Psi_{m-1}\rangle$ to 
the quotient $I_n(A)$ is an isomorphism. 
\end{proof}

\begin{lemma}\label{lem:char-poly}
Consider the $\phi$-linear endomorphism $F$ of $K^m$ with matrix
	$$\left(\begin{array}{llllll}
	0 & 0 & \hdots & & 0 & \mu_{m} \\
	1 & 0 & & & 0 & \mu_{m-1} \\ 
	0 & 1 & & & 0 & \mu_{m-2} \\
	\vdots  &  &\ddots & \ddots & \vdots  & \vdots  \\
	  &   & & &   &   \\
	0 & 0 & & & 1 & \mu_1
	\end{array}\right), $$
for some given $\mu_1,\dots,\mu_m\in K$. 
If $K^m$ admits an $R$-lattice which is stable under $F$, then we have $\mu_1,\dots,\mu_m\in R$.
\end{lemma}
\begin{proof}
The equal-characteristic analogue of this lemma is lemma 9.7 of \cite{Drin-us}. 
It was proven using two results in equal-characteristic Dieudonn\'e--Manin theory,
namely (B.1.5) and (B.1.9) of \cite{Laumon-book-vol1}, which are proved in pages 257--261.
All these arguments have straightforward translations to the mixed-characteristic setting,
and so we leave them to the reader. Or one could
refer to the original paper by Manin \cite{M1}, lemma 2.2 and theorem 2.2.
\end{proof}

\begin{theorem}
\label{intlam}
If $A$ splits at $m=2$, then $\lam \in R$.
\end{theorem}

\begin{proof}
Let $\gamma$ be the element associated to $\Theta_2$ as in proposition 
\ref{diff}.
We will prove the cases when $\gamma =0$ and $\gamma \ne 0$ separately.

\underline{Case $\gamma =0$}~:
When $\gamma =0$ we have $\mfrak{f}^* i^*(\Theta_2) = i^* \phi^*(\Theta_2)$.
Since $\Theta_2$ generates $\bX(A)_K$, we have $\mfrak{f}^* i^* = i^* \phi^*$ on all of $\bX(A)_K$.
Therefore for all $n \geq 1$, we have a $\phi$-linear map 
$\mfrak{f}^*:\bI_{n-1}(A) \map \bI_n(A)$ fitting in a morphism of exact sequences:
	$$
	\xymatrix{
		0 \ar[r] & \bX_n(A) \ar[r]^-{i^*} & \Hom(N^n,\hG) \ar^-\partial[r] & \bI_n(A) \ar[r] & 0 \\
		0 \ar[r] & \bX_{n-1}(A) \ar[u]_{\phi} \ar[r]^-{i^*} & 
\Hom(N^{n-1},\hG) \ar^-\partial[r] 
			\ar[u]_{\mfrak{f}^*} & \bI_{n-1}(A) \ar[r] \ar[u]_{\mfrak{f}^*} & 0 \\
	}
	$$
Recall $\bI(A)= \varinjlim \bI_n(A) \subset \Ext(A,\hG)$.
Then by proposition~\ref{pro:I-basis},
the vector space $\bI(A)_K$ has a $K$-basis $\partial\Psi_1$, and
with respect to this basis, the $\phi$-linear endomorphism $\mfrak{f}^*$ has 
matrix $\Gamma_0 = (\lam)$. 

Note that $\bI(A)$ is a finitely generated $R$-module since it is a 
submodule of $\Ext(A,\hG)$ which is a finitely generated free $R$-module and
$R$ is a discrete valuation ring.
Since $\Gamma_0$ is an endomorphism of $\bI(A)$
and hence an integral lattice of $\bI(A)_K$, we conclude that $\lam$ is 
integral, that is, $\lam \in R$.

\underline{Case $\gamma \ne 0$}~:
Recall $\bH(A)=\varinjlim \bH_n(A)$. Let us consider the matrix $\Gamma$ of the 
$\phi$-linear endomorphism
$\mfrak{f}$ of $\bH(A)_K$ with respect to the $K$-basis $\Psi_1,\Psi_2$ given by
 proposition~\ref{hn}. 
Then we have
	\begin{align*}
	i^* \phi^*\Theta_2 
	 &= \mfrak{f}^*(\Psi_{2})-\phi(\lam)\Psi_2 + \gamma\Psi_1.
	\end{align*}
Therefore we have 
	$$
	\mfrak{f}^*(\Psi_2) \equiv \phi(\lam)\Psi_2 - \gamma\Psi_1 \bmod
	i^*\phi^*(\bX_{2}(A)_{\phi})
	$$
and hence
	$$ \Gamma = 
\left(\begin{array}{ll}
0 & -\gamma \\
1 & \phi(\lam)
\end{array}\right)
	$$ 
We will now apply lemma~\ref{lem:char-poly} to the operator $\mfrak{f}^*$ on $\bH(A)_K$, but to do this we need
to produce an integral lattice $M$. Consider the commutative square
	$$
	\xymatrix{
	\bH(A) \ar^-{\Phi}[r]\ar[d] & \Ext^\sharp(A,\hG) \ar^j[d] \\
	\bH(A)_K \ar^-{\Phi_K}[r]  & \Ext^\sharp(A,\hG)_K.
	}
	$$
Let $M$ denote the image of $\bH(A)$ in $\bH(A)_K$. It is clearly stable under $\mfrak{f}^*$.
But also the maps $\Phi_K$ and $j$ are injective, by theorem~\ref{isocrys} ($\gamma$ being nonzero)
and because $\Ext^\sharp(A,\hG)\simeq R^r$;
so $M$ agrees with the image of $\bH(A)$ in $\Ext^\sharp(A,\hG)$ and is therefore finitely generated.

We can then apply lemma~\ref{lem:char-poly} and deduce $\phi(\lam) \in R$.
This implies $\lam\in R$, since $R/\pi R$ is a field and hence 
the Frobenius map on it is injective.
\end{proof} 

\begin{corollary}
\label{cor:square-iso}
\begin{enumerate}
	\item The element $\Theta_m\in\bX_m(A)_K$ lies in $\bX_m(A)$.
	\item For $n\geq m$, all the maps in the diagram 
	$$
	\xymatrix{
	\bX_n(A)_{\phi}/\bX_{n-1}(A)_{\phi} \ar^{\phi^*}[r]\ar^{i^*}[d] & \bX_{n+1}(A)/
\bX_n(A) \ar^{i^*}[d] \\
	\Hom(N^n,\hG)_{\phi}/\HomA(N^{n-1},\hG)_{\phi} \ar^{\mfrak{f}^*}[r] & \HomA(N^{n+1},\hG)/\HomA(N^n,\hG)
	}
	$$
are isomorphisms.
\end{enumerate}
\end{corollary}

\begin{proof}
(1): By theorem~\ref{intlam}, the element $i^*\Theta_m$ of $\HomA(N^m,\hG)_K$ actually lies
in $\HomA(N^m,\hG)$, and therefore  by the exact sequence~(\ref{sel}) we have $\Theta_m\in\bX_m(A)$.

(2): By proposition~\ref{pro-filt}, we know $\mfrak{f}^*$ is an isomorphism
and the maps $i^*$ are injective. 
So to show they are isomorphisms for all $n\geq m$, it is enough
to show they are surjective. The $R$-linear generator $\Psi_m$ of $\HomA(N^m,\hG)/\HomA(N^{m-1},\hG)$
is the image of $\Theta_m$, which by part (1), lies in $\bX_m(A)$. Therefore $i^*$ is surjective
for $n=m$. Then because $\mfrak{f}^*$ is an isomorphism,
it follows by induction that $i^*$ is surjective for all $n\geq m$.

Finally, $\phi^*$ is an isomorphism because all the other morphisms in the diagram are.
\end{proof}

By definition, $i^*(\phi^j)^*\Theta_m$ is $\Psi_{m+j}$ plus 
lower order characters with coefficients in $K$. 
However, the implication of the corollary above says
that the coefficients of the lower order characters are in fact 
integral. Now the following result proves theorem \ref{phigen-intro}.

\begin{theorem}
\label{phigen-body}
	Let $A$ be an elliptic curve that splits at $m$. 
	\begin{enumerate}
		\item For any $n\geq m$, the composition
			\begin{equation}
				\label{map-X}
				\bX_n(A) \longmap \HomA(N^n,\hG) \longmap \HomA(N^n,\hG)/\HomA(N^{m-1},\hG)				
			\end{equation}
			is an isomorphism of $R$-modules. 
\item $\bX_n(A)$ is freely generated as an $R$-module by 
$\Theta_m,\dots,(\phi^*)^{n-m}\Theta_m$.

\item  We have 
$$\bH(A) \simeq \left\{\begin{array}{ll} 
R\langle \Psi_1 \rangle, & \mb{if } m=1 \mb{,  i.e. $A$ is a canonical lift}\\
R\langle \Psi_1,\Psi_2 \rangle, & \mb{if } m=2
\end{array}\right.$$
	\end{enumerate}
\end{theorem}
\begin{proof}
(1): By corollary~\ref{cor:square-iso}, the induced morphism on each graded piece is an isomorphism
of $R$-modules.
It then follows that the map in question is also an isomorphism of $R$-modules.

(2): This follows formally from (1) and the fact, which follows from  corollary \ref{cor:square-iso}, that 
the map (\ref{map-X}) sends any $(\phi^*)^j\Theta_m$ to $\Psi_{m+j}$ plus lower order terms.

(3): Recall we have $\bX_m(A) \simeq R\langle \Theta_m \rangle$ where 
$i^*\Theta_m \equiv \Psi_m \bmod \Hom(N^{m-1},\hG)$.
By (2) as above, we have $\bX_n(A) \simeq R\langle\Theta_m,\phi^*\Theta_m,
\dots, (\phi^*)^{n-m}\Theta_m\rangle$.
By proposition \ref{diff}, for all $n \geq m$ we have 
\begin{align}
\nonumber i^* (\phi^*)^{n-m}\Theta_m &\equiv
 ({\mfrak{f}^*})^{n-m}i^*\Theta_m \bmod \Hom(N^1,\hG) \\
&\equiv \Psi_n \bmod \Hom(N^{n-1},\hG) \nonumber
\end{align}
Therefore we have $\bH_n(A) \simeq \Hom(N^m,\hG)$ for all $n \geq m+1$ and 
hence the limit $\bH(A)$ is isomorphic to $\Hom(N^m,\hG)$ and we are done.
\end{proof}

\subsection{The integral $F$-crystal $\bH(A)$ for an elliptic curve $A$}
\label{subsec-intH}

The filtration of the isocrystal $\bH(A)$ is given by 
\begin{equation}
\label{filtration-H}
\bH(A)^i = \left\{\begin{array}{ll}
\bH(A), & \mb{ if } i \leq 0 \\
\bXp(A), & \mb{ if } i \geq 1
\end{array}\right.
\end{equation}

\begin{theorem}
\label{elliptic-crystal}
Let $A$ be an elliptic curve over $S$. 

$(1)$ If $m=1$, that is $A$ is a canonical lift, then $$\bH(A) = \bXp(A) 
\simeq R\langle \Psi_1 \rangle.$$ 
The semilinear operator $\mfrak{f}^*$ acts as
 $\mfrak{f}^*(\Psi_1) = \gamma \Psi_1$.

$(2)$ If $m=2$, then 
$$\bH(A) \simeq R\langle\Psi_1,\Psi_2\rangle,~ \bXp(A) \simeq R\langle 
\Theta_2 \rangle$$ 
The semilinear operator $\mfrak{f}^*$ acts on the $R$-basis of $\bH(A)$ as
$$\mfrak{f}^*(\Psi_1) = \Psi_2,~ \mfrak{f}^*(\Psi_2)= \phi(\lambda) \Psi_2 -
\gamma \Psi_1.$$
where $\pi\mid \gamma$.

Moreover, in both the above cases of $m$, if $\gamma \not\equiv 0 \bmod \pi^2$, 
then $\bH(A)$ is a weakly admissible filtered isocrystal of dimension $1$ and
$2$ respectively. 

\end{theorem}

\begin{proof}
By theorem
\ref{phigen-body} (3), we have 
$$\bH(A) \simeq \left\{\begin{array}{ll} 
R\langle \Psi_1 \rangle, & \mb{if } m=1 \mb{,  i.e. $A$ is a canonical lift},\\
R\langle \Psi_1,\Psi_2 \rangle, & \mb{if } m=2.
\end{array}\right.$$
We also have isomorphisms for $n\geq m$
        $$
        R\langle \Theta_m\rangle = \bXp(A) \longisomap \bX_n(A)/\phi^*(\bX_{n-1}(A)_{\phi}).
        $$
The filtration on
$\bH(A)$ is given by $\bH(A) \supseteq \bXp(A)$.

In the case when $m=2$, 
the action of the semi-linear operator $\mfrak{f}^*$ with respect to the above 
choice of basis $\Psi_1$ and $\Psi_2$ of $\bH(A)$ 
 is described by the matrix 
$$ \Gamma = 
\left(\begin{array}{ll}
0 & -\gamma \\
1 & \phi(\lam)
\end{array}\right)
	$$ 
as in the proof of theorem~\ref{intlam}. This proves the action of $\mfrak{f}^*$
on the $R$-basis $\Psi_1$ and $\Psi_2$ is as stated. 

In the case when $m=1$, it is straightforward to see that $\bH(A)$ is 
weakly admissible when $\gamma \not\equiv 0 \bmod  \pi^2$ by definition 
(\ref{weakadm}).

In the case when $m=2$, 
when $\gamma \not\equiv 0 \bmod \pi^2$, then $\mfrak{f}^*$ does not preserve
the $R$-submodule $\bXp(A) \subset \bH(A)$ and hence makes $\bH(A)$ 
weakly admissible.
\end{proof}

Combining these, we have the following map between exact sequences of $R$-modules, as in (\ref{diag-crys}):
        $$
        \xymatrix{
        0 \ar[r] & 
                \bXp(A) \ar[r] \ar[d]_-\Upsilon & 
                \bH(A) \ar[r] \ar[d]^-{\Phi} & 
                \bI(A) \ar[r] \ar@{^{(}->}[d]& 
                0\\
        0 \ar[r] &
                \Lie(A)^* \ar[r] & 
                \bH_{\mathrm{dR}}(A) \ar[r] & 
                \Ext(A,\hG) \ar[r] & 
                0
        }
        $$
where $\Upsilon$ sends $\Theta_m$ to $\gamma/\pi$ (in coordinates),
and $\Phi$ is injective if and only if $\gamma\neq 0$. However,
we do note that the map $\Phi$ is not compatible between the two $F$-crystal 
structure on $\bH(A)$ and the crystalline structure on $\bH_{\mathrm{dR}}(A)$.

\footnotesize{

}
\end{document}